\documentclass[journal]{IEEEtran}
\usepackage{amsmath,amsfonts}
\usepackage{algorithmic}
\usepackage{subfigure}
\usepackage{textcomp}
\usepackage{stfloats}
\usepackage{url}
\usepackage{verbatim}
\usepackage{graphicx}
\usepackage{mathtools}
\usepackage{mathdots}
\usepackage{diagbox}
\def\BibTeX{{\rm B\kern-.05em{\sc i\kern-.025em b}\kern-.08em
    T\kern-.1667em\lower.7ex\hbox{E}\kern-.125emX}}
\usepackage{balance}
\usepackage{hyperref}
\usepackage{xcolor}
\usepackage{bbold}
\usepackage{amsthm,amssymb}
\usepackage{lipsum}
\usepackage{multicol}
\usepackage{array}
\usepackage[auth-sc]{authblk}
\usepackage{footnote}

\newcolumntype{P}[1]{>{\centering\arraybackslash}p{#1}}

\definecolor{Mgreen}{RGB}{80, 175, 80}
\definecolor{Mblue}{RGB}{80, 80, 200}
\definecolor{Mred}{RGB}{220, 50, 50}

\def \A {\mathbf{A}}

\def \B {\mathbf{B}}

\def \D {\mathbf{D}}

\def \e {\mathbf{e}}

\def \G {\mathbf{G}}

\def \H {\mathbf{H}}
\def \I {\mathbf{I}}

\def \J {\mathbf{J}}

\def \K {\mathbf{K}}

\def \Q {\mathbf{Q}}

\def \R {\mathbf{R}}
\def \s {\mathbf{s}}
\def \S {\mathbf{S}}

\def \U {\mathbf{U}}
\def \u {\mathbf{u}}

\def \W {\mathbf{W}}
\def \w {\mathbf{w}}
\def \x {\mathbf{x}}
\def \X {\mathbf{X}}
\def \Y {\mathbf{Y}}
\def \y {\mathbf{y}}

\def \z {\mathbf{z}}

\def \Acal {\mathcal{A}}

\def \Ccal {\mathcal{C}}
\def \Dcal {\mathcal{D}}

\def \Hcal {\mathcal{H}}

\def \Kcal {\mathcal{K}}

\def \Ncal {\mathcal{N}}
\def \Ocal {\mathcal{O}}

\def \Zcal {\mathcal{Z}}

\def \Cbb {\mathbb{C}}
\def \Ebb {\mathbb{E}}

\def \Nbb {\mathbb{N}}
\def \Pbb {\mathbb{P}}
\def \Rbb {\mathbb{R}}

\def \Vbb {\mathbb{V}}
\def \Zbb {\mathbb{Z}}

\def \drm {\mathrm{d}}

\def \irm {\mathrm{i}}
\def \Prm {\mathrm{P}}

\def \betabs {\boldsymbol{\beta}}

\def \gammabs {\boldsymbol{\gamma}}

\def \mubs {\boldsymbol{\mu}}

\def \zetabs {\boldsymbol{\zeta}}

\def \Deltabs {\boldsymbol{\Delta}}
\def \Gammabs {\boldsymbol{\Gamma}}

\def \Omegabs {\boldsymbol{\Omega}}

\def \Thetabs {\boldsymbol{\Theta}}
\def \Upsilonbs {\boldsymbol{\Upsilon}}
\def \Xibs {\boldsymbol{\Xi}}

\def \det {\mathrm{det}}

\def \Tr {\mathrm{tr}}
\def \tr {\mathrm{tr}}

\def \diag{\mathrm{diag}}
\def \bdiag{\mathrm{bdiag}}

\DeclareMathOperator{\supp}{supp}

\newtheorem{assumption}{Assumption}
\newtheorem{corollary}{Corollary}
\newtheorem{theorem}{Theorem}
\newtheorem{proposition}{Proposition}
\newtheorem{lemma}{Lemma}
\newtheorem{remark}{Remark}

\renewcommand{\Im}{\mathrm{Im}}
\renewcommand{\Re}{\mathrm{Re}}
\begin{document}

%%%%%%%%%%%%%%%%%%%%%%%%%%%%%%%%%%%%%%%%%%%%%%%%%%%%%%%%%%%%%%%%%%%%%%%%%%%%%%%%%%%%%%%%%%%%
%%%%%%%%%%%%%%%%%%%%%%%%%%%%%%%%%%%%%%%%%%%%%%%%%%%%%%%%%%%%%%%%%%%%%%%%%%%%%%%%%%%%%%%%%%%%
%%%%%%%%%%%%%%%%%%%%%%%%%%%%%%%%%%%%%%%%%%%%%%%%%%%%%%%%%%%%%%%%%%%%%%%%%%%%%%%%%%%%%%%%%%%%
%%%%%%%%%%%%%%%%%%%%%%%%%%%%%%%%%%%%%%%%%%%%%%%%%%%%%%%%%%%%%%%%%%%%%%%%%%%%%%%%%%%%%%%%%%%%
%%%%%%%%%%%%%%%%%%%%%%%%%%%%%%%%%%%%%%%%%%%%%%%%%%%%%%%%%%%%%%%%%%%%%%%%%%%%%%%%%%%%%%%%%%%%
%%%%%%%%%%%%%%%%%%%%%%%%%%%%%%%%%%%%%%%%%%%%%%%%%%%%%%%%%%%%%%%%%%%%%%%%%%%%%%%%%%%%%%%%%%%%
%\title{A New Test Statistic for Change Detection in the Covariance Structure of High-Dimensional Gaussian Low-Rank Models}
\title{A New Statistic for Testing Covariance Equality in High-Dimensional Gaussian Low-Rank Models}
\author
{
    R. Beisson, \IEEEmembership{Student Member, IEEE}, 
    P. Vallet, \IEEEmembership{Member, IEEE}, 
    A. Giremus, \IEEEmembership{Member, IEEE}, 
    G. Ginolhac, \IEEEmembership{Member, IEEE}
    \thanks{R. Beisson, P. Vallet and A. Giremus are with Laboratoire de l'Int{\'e}gration du Mat{\'e}riau au Syst{\`e}me (CNRS, Univ. Bordeaux, Bordeaux INP), 351 Cours de la Libération, 33400 Talence (France)}
    \thanks{G. Ginolhac is with Laboratoire d'Informatique, Syst{\`e}mes, Traitement de l'Information et de la Connaissance (Univ. Savoie/Mont-Blanc, Polytech Annecy),  5 chemin de Bellevue, 74940 Annecy (France)}
    \thanks{This work was partially supported by Agence de l'Innovation de D{\'e}fense and R{\'e}gion Nouvelle-Aquitaine. The material of this paper was partly presented in the conference paper  \cite{Beisson2021SSP}.}
}

\maketitle

\begin{abstract}
In this paper, we consider the problem of testing equality of the covariance matrices of $L$ complex Gaussian multivariate time series of dimension $M$. 
We study the special case where each of the $L$ covariance matrices is modeled as a rank $K$ perturbation of the identity matrix, corresponding to a signal plus noise model. 
A new test statistic based on the estimates of the eigenvalues of the different covariance matrices is proposed. In particular, we show that this statistic is consistent and with controlled type I error in the high-dimensional asymptotic regime where the sample sizes $N_1, \dots , N_L$ of each time series and the dimension $M$ both converge to infinity at the same rate, while $K$ and $L$ are kept fixed. 
We also provide some simulations on synthetic and real data (SAR images) which demonstrate significant improvements over some classical methods such as the GLRT, or other alternative methods relevant for the high-dimensional regime and the low-rank model.
\end{abstract}

\section{Introduction}
\label{section:Introduction}

\IEEEPARstart{D}{etecting} changes in the behaviour of multivariate time series is a fundamental problem in many applications going from remote sensing \cite{Conradsen2003,Ciuonzo2017,Mian2018,Mian2020} and wireless communications \cite{Liu2021}  to finance \cite{Galeano2007}, climatology \cite{Ribes2009} or genomics \cite{Zhou2019}. In several of those applications, a usual approach consists in modeling the changes using the distribution of the time series, and in particular through an evolution in the structure of the covariance matrix.

\par Consider the context of $M$-dimensional time series $\left(\mathbf{y}_{n,1}\right)_{n \in \mathbb{Z}}, \ldots, \left(\mathbf{y}_{n,L}\right)_{n \in \mathbb{Z}}$, assumed mutually independent and such that for all $\ell \in \{1, \ldots, L\}$,
\begin{equation}
    \left(\y_{n,\ell}\right)_{n \in \Zbb} \stackrel{\text{i.i.d.}}{\sim} \Ncal_{\Cbb^M}\left(\mathbf{0},\R_{\ell}\right),
\end{equation}
where $\Ncal_{\Cbb^M}\left(\mathbf{0},\R_{\ell}\right)$ denotes the zero-mean complex normal distribution with covariance matrix $\R_{\ell}$. Detecting the changes in the distribution of $\left(\mathbf{y}_{n,\ell}\right)_{n \in \mathbb{Z}}$, for all $\ell \in \{1, \ldots, L\}$, can be formalized as the following binary hypothesis test dealing with the equality of the $L$ covariance matrices $\R_1,\ldots,\R_L$,
\begin{align}
    \label{eq:General_Test}
    \begin{split}
      \mathcal{H}_0 :& \quad \mathbf{R}_1 = \ldots = \mathbf{R}_L \\
      \mathcal{H}_1 :& \quad \exists (i,j) \in \{1, \dots, L\}^2: \mathbf{R}_i \neq \mathbf{R}_{j}
    \end{split}.
\end{align}
Assume that for all $\ell \in \{1, \ldots, L\}$, $N_\ell$ observations $\mathbf{y}_{1,\ell}, \dots, \mathbf{y}_{N_\ell,\ell}$ are available and let $N = N_1 + \dots + N_L$. A large class of test statistics widely encountered in the literature \cite{Ciuonzo2017} involves, provided that $M < N_1,\ldots, N_L$, linear spectral statistics of the matrices $\hat{\R}_{\ell}^{-1}\hat{\R}$ of the form:
\begin{equation}
    \label{eq:Linear_Spectral_Statistic_F_Matrix}
    S =  \sum_{\ell=1}^L \frac{N_{\ell}}{N} \frac{1}{M}\sum_{k=1}^M  \varphi\left(\lambda_k(\hat{\R}_{\ell}^{-1}\hat{\R})\right),
\end{equation}
where $\lambda_k(\hat{\R}_{\ell}^{-1}\hat{\R})$, for all $k \in \{1,\ldots,M\}$, are the eigenvalues of the matrix $\hat{\R}_{\ell}^{-1}\hat{\R}$ with
\begin{equation}
  \hat{\mathbf{R}}_{\ell} \vcentcolon= \frac{1}{N_{\ell}} \sum_{n=1}^{N_{\ell}} \mathbf{y}_{n,\ell} \mathbf{y}_{n,\ell}^*,
\end{equation}
denoting the sample covariance matrix (SCM) associated with $\mathbf{R}_\ell$ and
\begin{equation}
    \hat{\mathbf{R}} = \sum_{\ell=1}^L \frac{N_{\ell}}{N} \hat{\mathbf{R}}_{\ell}.
\end{equation}
In \eqref{eq:Linear_Spectral_Statistic_F_Matrix}, $\varphi$ denotes some continuous function defined on $\left(0 , +\infty \right)$. In particular, the Generalized Likelihood Ratio (GLR) \cite{Ciuonzo2017} with $\varphi(x) = \log(x)$ or the \textit{Nagao} statistic with $\varphi(x) = (x - 1)^2$ are included in the class of statistics \eqref{eq:Linear_Spectral_Statistic_F_Matrix}. The presence of a change in the covariance is decided by comparing \eqref{eq:Linear_Spectral_Statistic_F_Matrix} to a threshold $\epsilon$ chosen to guarantee a certain type I error and the null hypothesis $\mathcal{H}_0$ is rejected if $S>\epsilon$. Moreover, the test statistics based on \eqref{eq:Linear_Spectral_Statistic_F_Matrix} have the key property that the distribution of $S$ under $\mathcal{H}_0$ is independent of $\mathbf{R}_1 = \ldots = \mathbf{R}_L$, which allows to control its type I error. 

However, in practice, the distribution of statistics of type \eqref{eq:Linear_Spectral_Statistic_F_Matrix} under $\Hcal_0$ is untractable and only known in a few special cases for finite $M,N_1,\ldots,N_L$ (e.g. for the GLR, see \cite{gupta1984distribution}). To circumvent this issue, approximations in the \textit{low-dimensional} (or \textit{large sample size}) regime in which $N_1,\ldots,N_L \to \infty$ while $M,L$ are fixed can be derived, see e.g. \cite[Th. 10.8.4]{Muirhead1982}. While the latter are meant to be used in practical scenarios where $N_1,\ldots,N_L \gg M$, they may not be reliable in contexts involving high-dimensional (large $M$) observations or moderate sample sizes $N_1,\ldots,N_L$. Indeed, in that high-dimensional case, it is often more reasonable to assume that $M, N_1,\ldots,N_L$ are of the same order of magnitude in which case the predictions of the distribution of \eqref{eq:Linear_Spectral_Statistic_F_Matrix} under $\Hcal_0$ in the \textit{low-dimensional} regime become irrelevant.

The context where $M,N_1,\ldots,N_L$ are of the same order of magnitude can be modeled more realistically by the \textit{high-dimensional regime} in which it is assumed that $M$ converges to infinity together with $N_1,\ldots,N_L$ such that $\frac{M}{N_{\ell}} \to c_{\ell} > 0$, while $L$ is kept fixed. In this non-standard regime, the asymptotic distribution of the statistic $S$ can be derived using random matrix theory techniques (see e.g. \cite{Zheng2012} for the case $L=2$).  

Moreover, in several applications involving high-dimensional observations, the potential changes in the covariance $\mathbf{R}_\ell$ may only be carried by a low-rank component (see e.g. \cite{Johnstone2001, combernoux2018performance,Vallet2015}). This is the case, e.g., in array processing when dealing with a large array of $M$ sensors and a small number $K$ of source signals compared to $M$ \cite{Vallet2015}.

In that case, we have the model
\begin{equation}
    \label{eq:Low_Rank_Model}
    \mathbf{R}_{\ell} = \boldsymbol{\Gamma}_{\ell} + \sigma^2 \mathbf{I},
\end{equation}
with $\boldsymbol{\Gamma}_{\ell}$ the covariance matrix of rank $K<M$ of a useful signal and $\sigma^2\mathbf{I}$ the covariance matrix of a spatially white additive noise.
When the rank $K$ remains constant in the \textit{high-dimensional regime}, the matrices $\R_{\ell}^{-1}\R_{\ell'}$ are fixed rank perturbations of the identity. 
Using well-known results \cite{Wachter1980, Wang2017} on the asymptotic spectral distribution of the Fisher type random matrices $\hat{\R}_{\ell}^{-1}\hat{\R}$, one can show under both $\Hcal_0$ and $\Hcal_1$ that
\begin{align}
  S \xrightarrow[]{} \sum_{\ell=1}^L \frac{c}{c_{\ell}} \int_{x_{\ell}^-}^{x_{\ell}^+} \varphi\left(\frac{c}{c_{\ell}}(1+x)\right) f_{\ell}(x) \drm x,
\end{align}
almost surely (a.s.) where $c = (c_1^{-1} + \ldots + c_L^{-1})^{-1}$ and where $f_{\ell}$ is the so-called \textit{Wachter} distribution given by
\begin{align}
  f_{\ell}(x) = \left(\frac{1}{c_{\ell}} -1\right) \frac{\sqrt{(x-x_{\ell}^-) (x_{\ell}^+ - x)}}{2 \pi x (1+x)} \mathbb{1}_{[x_{\ell}^-,x_{\ell}^+]}(x),
\end{align}
with $x_{\ell}^\pm = \frac{c_\ell - c}{c (1-c_{\ell})^2} \left(1 \pm \sqrt{c_{\ell} + \frac{c c_{\ell}}{c_{\ell} - c} - \frac{c c_{\ell}^2}{c_{\ell} - c}}\right)^2$.
Thus $S$ converges to the same limit under both hypotheses $\Hcal_0$ and $\Hcal_1$,  which indicates that test statistics relying on \eqref{eq:Linear_Spectral_Statistic_F_Matrix} might not be relevant in the \textit{high-dimensional regime} and for the low-rank model in \eqref{eq:Low_Rank_Model}.

So far from our knowledge, the problem of covariance equality testing under low-rank models has not received much attention in the literature. The work of \cite{Abdallah2019} considers the GLRT, under a low-rank Gaussian model, for a covariance equality test with a different alternative hypothesis $\Hcal'_1: \R_1 \neq \R_2 = \ldots = \R_L$. An extension to the specific case of subspace equality test has also been proposed by the same authors in \cite{abdallah2019signal}.

Under the model \eqref{eq:Low_Rank_Model}, the information about a potential change is contained in the $K$ largest eigenvalues and associated eigenvectors of $\R_{\ell}$. Therefore, classical results on the \emph{spiked models} for random matrices of the Fisher type \cite{Wang2017} can be exploited to characterize the asymptotic behaviour of the extreme eigenvalues of $(\hat{\R}_{\ell}^{-1}\hat{\R})_{\ell=1,\ldots,L}$, from which information about a potential change can be extracted. 
In the same way, the asymptotic behaviour of the largest eigenvalues of the spiked Wishart-type matrices \cite{Benaych-Georges2011} $\hat{\R}, (\hat{\R}_{\ell})_{\ell=1,\ldots,L}$ convey information about changes in the true covariances $\R_1,\ldots,\R_L$ \cite{Beisson2021SSP}, which can also be exploited to build test statistics relevant for the low-rank model in the high-dimensional regime. This latter option is the path followed in this paper.

\par \textit{Contributions.}
In this paper, we derive a new test statistic, no longer based on the family of statistics $S$ studied in \cite{Ciuonzo2017}, but which relies on the $K$ largest eigenvalues of the matrices $\hat{\mathbf{R}}_1, \dots, \hat{\mathbf{R}}_L, \hat{\mathbf{R}}$. More precisely, the test statistic compares in a certain sense estimates of the eigenvalues of the matrices $\Gammabs_1,\ldots,\Gammabs_L$ with estimates of the eigenvalues of the mixture $\Gammabs = \sum_{\ell=1}^L \frac{N_{\ell}}{N} \Gammabs_{\ell}$. We show that the proposed test statistic is consistent under the high-dimensional regime and the low-rank model \eqref{eq:Low_Rank_Model}, and with a controlled asymptotic type I error. To that purpose, the results of \cite{Benaych-Georges2011a}, which provides the asymptotic distribution of the $K$ largest eigenvalues of $\hat{\R}_{\ell}$ for a fixed $\ell$, are extended to provide the joint asymptotic distribution of the $K$ largest eigenvalues of the matrices $\hat{\R},\hat{\R}_1,\ldots,\hat{\R}_L$. The proposed test statistic is then compared to various alternatives, including the GLRT for the low-rank model \eqref{eq:Low_Rank_Model} as well as a statistic built from the results of \cite{Wang2017} on the extreme eigenvalues of the spiked Fisher matrices  $(\hat{\R}_{\ell}^{-1}\hat{\R})_{\ell=1,\ldots,L}$. We also provide an empirical study of the proposed test statistic on Synthetic Aperture Radar (SAR) images for detecting changes between two scenes.

\par\textit{Organization.} The paper is organized as follows. In Section \ref{section:Spectrum_of_hat_R}, we study an extension of the results of \cite{Benaych-Georges2011a} on the asymptotic distribution of the largest eigenvalues of $\hat{\mathbf{R}}_1, \dots, \hat{\mathbf{R}}_L, \hat{\mathbf{R}}$. In Section \ref{section:Test_Statistic}, we exploit the results derived in the previous section to build a new test statistic, for which we study its performance in the \textit{high-dimensional regime}. Sections \ref{section:alt_methods} and \ref{section:Simulations} are dedicated to compare, both theoretically and numerically, our proposed test statistic with alternative approaches. Simulations on synthetic data and on real data (SAR images) are provided.

\par \textit{Notations.} For $a\in \mathbb{R}$, $a^+$ denotes the positive part. Vectors and matrices are denoted with boldface lower case and upper case letters respectively. For a complex matrix $\A$, we denote by $\A^T$ and $\textbf{A}^*$ its transpose and conjugate transpose. If $\A$ is a $n \times n$ complex matrix, $\Tr(\A)$ denotes its trace and $\lambda_1(\A),\ldots,\lambda_n(\A)$ denote its eigenvalues. If $\A$ is Hermitian, the eigenvalues are considered in decreasing order $\lambda_1(\A) \geq \ldots \geq \lambda_n(\A)$. For matrices $\A_1,\ldots,\A_n$, $\bdiag(\A_1,\ldots,\A_n)$ denotes the the block diagonal matrix formed by $\A_1,\ldots,\A_n$. The complex circular Gaussian distribution on $\Cbb^n$ with covariance matrix $\R$ is denoted as $\Ncal_{\Cbb^n}(\mathbf{0},\R)$, while the real Gaussian distribution on $\Rbb^n$ with mean $\mubs$ and covariance matrix $\R$ is denoted as $\Ncal_{\Rbb^n}(\mubs,\R)$.

\section{Spectrum of \texorpdfstring{$\hat{\mathbf{R}}$}{R}}
\label{section:Spectrum_of_hat_R}

\par In this section, we study the asymptotic behavior at $1^{st}$ and $2^{nd}$ orders of the largest eigenvalues of $\hat{\mathbf{R}}$ when the matrices $\hat{\mathbf{R}}_1, \dots, \hat{\mathbf{R}}_L$ follow the low-rank model $\eqref{eq:Low_Rank_Model}$.
\par Consider the following two assumptions, which describe the \textit{high-dimensional regime} and specify the asymptotic behavior of the eigenvalues of $\boldsymbol{\Gamma}_1, \dots,\boldsymbol{\Gamma}_L$.
\begin{assumption}
    \label{assumption:Regime_Grandes_Dimensions}
    The sample sizes $N_1 = N_1(M), \ldots, N_L = N_L(M)$ are functions of $M$ such that
    \begin{equation}
        \frac{M}{N_{\ell}} = c_{\ell} + \textit{o}\left(\frac{1}{\sqrt{M}}\right),
    \end{equation}
    as $M \to \infty$, where $c_1, \dots, c_L$ $> 0$ and $K, L$ are independent of $M$.
\end{assumption}
In comparison with the classical \emph{low-dimensional regime} where $M$ is assumed fixed while $N_1,\ldots,N_L \to \infty$ \ (see e.g. \cite{Anderson1958, Muirhead1982}), the high-dimensional regime described in Assumption \ref{assumption:Regime_Grandes_Dimensions} models practical scenarios where the sample sizes $N_1,\ldots,N_L$ are of the same order of magnitude as the dimension $M$ and where $K$ is small compared to $M$. This regime has been widely used in the high-dimensional statistics literature (see e.g. \cite{johnstone_high_2006}), as well as in the signal processing applications (see e.g. \cite{Vallet2015, Mestre2017, couillet_robust_2015}).

In what follows the \textit{high-dimensional regime} described in Assumption \ref{assumption:Regime_Grandes_Dimensions} is represented by the notation $M \to \infty$. We also define 
\begin{align}
   c \vcentcolon= \left(\sum_{\ell = 1}^L \frac{1}{c_\ell}\right)^{-1},
\end{align}
as well as
\begin{equation}
    \label{eq:Gamma}
    \boldsymbol{\Gamma} \vcentcolon= \sum_{\ell=1}^L \frac{N_{\ell}}{N} \boldsymbol{\Gamma}_{\ell}.
\end{equation}
One can notice that $\mathbf{\Gamma}$ is the the pooling of the low-rank covariance matrices $\boldsymbol{\Gamma}_1, \dots,\boldsymbol{\Gamma}_L$ and has rank at most $KL$.
In the following, we also need to ensure the convergence of the eigenvalues of matrices $\Gammabs_{\ell}, \Gammabs$ in the high-dimensional regime.
\begin{assumption}
    \label{assumption:Convergence_gamma}
    For all $k \in \{1,\ldots,K\}$, $\ell \in \{1, \dots, L\}$, 
    \begin{equation}
        \lambda_{k}(\boldsymbol{\Gamma}_{\ell}) = 
        \gamma_{k,\ell} + \textit{o}\left(\frac{1}{\sqrt{M}}\right),
    \end{equation}
    and for all $k \in \{1, \dots, KL\}$,
    \begin{equation}
        \lambda_{k}(\boldsymbol{\Gamma}) = \gamma_k + \textit{o}\left(\frac{1}{\sqrt{M}}\right).
    \end{equation}
\end{assumption}
We note that Assumption \ref{assumption:Convergence_gamma} is a purely technical assumption which is not restrictive in practice as the corresponding results derived from it are meant to be used for fixed values of $M,N,K$.

Under Assumptions \ref{assumption:Regime_Grandes_Dimensions} and \ref{assumption:Convergence_gamma}, the global behaviour of the eigenvalues of $\hat{\R}$ can be described through its empirical spectral distribution defined as the random probability measure
\begin{equation}
    \hat{\mu} = \frac{1}{M} \sum_{k=1}^M \delta_{\lambda_k\left(\hat{\R}\right)},
\end{equation}
where $\delta_x$ is the Dirac measure centered at $x$. Under the model \eqref{eq:Low_Rank_Model}, each covariance matrix $\mathbf{R}_1, \dots, \mathbf{R}_L$ is a fixed rank $K$ perturbation of the matrix $\sigma^2 \mathbf{I}$ and it can be shown using standard perturbations arguments that $\hat{\mu}$ asymptotically behaves as the Marcenko-Pastur distribution, i.e. $\hat{\mu}$ converges weakly almost surely \textit{(a.s.)} to the probability measure:
\begin{align}
    &\mu(\drm x) = \frac{\sqrt{(x-x^-)(x^+-x)}}{2\pi\sigma^2 c x} \mathbb{1}_{[x^-,x^+]}(x) \drm x \notag \\
    &\qquad \quad + \left(1-\frac{1}{c}\right)^+ \delta_0(\drm x),
    \label{eq:Marchenko_Pastur_distribution}
\end{align}
where $x^\pm = \sigma^2(1\pm\sqrt{c})^2$. Consequently, any functional of the type
\begin{align}
    \hat{\mu}(\varphi) \vcentcolon= \frac{1}{M} \sum_{k=1}^M \varphi(\lambda_k(\hat{\R})),
    \label{eq:lss_scm}
\end{align}
where $\varphi$ is a bounded continuous function, satisfies
\begin{equation}
    \hat{\mu}(\varphi) =
    \int_{\Rbb} \varphi(\lambda) \drm \hat{\mu}(\lambda)
    \xrightarrow[M\to\infty]{a.s.} \int_{\Rbb} \varphi(\lambda) \drm \mu(\lambda).
    \label{eq:Linear_Spectral_Statistic_Rhat}
\end{equation}
As the limit in \eqref{eq:Linear_Spectral_Statistic_Rhat} only depends on $\sigma^2$ and $c$, it is not possible to recover information on the low-rank matrices $\boldsymbol{\Gamma}_1, \dots,\boldsymbol{\Gamma}_L$ in the \textit{high-dimensional regime} from statistics of type \eqref{eq:lss_scm}. However, under the previous assumptions, it can be shown that the information related to the spectrum of $\boldsymbol{\Gamma}$ can be found in the largest $KL$ eigenvalues of $\hat{\mathbf{R}}$, thanks to the following result.
\begin{theorem}
    \label{theorem:Spike_model_Limits}
    Under Assumptions \ref{assumption:Regime_Grandes_Dimensions} and \ref{assumption:Convergence_gamma},  $\forall k \in \{1,\ldots,KL\}$,
    \begin{equation}
        \lambda_k\left(\hat{\mathbf{R}}\right) \xrightarrow[M\to\infty]{a.s.} \phi_c\left(\gamma_k,\sigma^2\right),
    \end{equation}
    with 
    \begin{equation}
        \phi_{c}(\gamma,\sigma^2) \vcentcolon=
        \begin{cases}
        \frac{(\gamma+\sigma^2)(\gamma+\sigma^2 c)}{\gamma} & \text{ if } \gamma > \sigma^2 \sqrt{c}
        \\
        \sigma^2 (1+\sqrt{c})^2 & \text{ if } \gamma \leq \sigma^2 \sqrt{c}
        \end{cases},
    \end{equation}
    Moreover, $\lambda_{KL+1}(\hat{\mathbf{R}}) \to \sigma^2 \left(1 + \sqrt{c}\right)^2$ $a.s.$ when $M \to \infty$.
\end{theorem}
\begin{proof}
    The proof of Theorem \ref{theorem:Spike_model_Limits} is deferred to Appendix \ref{section:Proof_Theorem_Spike_Model_Limits}. 
    % \textit{Theorem \ref{theorem:Spike_model_Limits} proof, detailed in Appendix \ref{section:Proof_Theorem_Spike_Model_Limits}, uses the techniques presented in \cite{Benaych-Georges2011a} applied for \textit{Wishart}-type random matrix models.}
\end{proof}
The matrix $\hat{\R}$ being a mixture of $L$ independent but not identically distributed Wishart matrices, we note that Theorem \ref{theorem:Spike_model_Limits} provides an extension of the results of \cite[Th. 2.7]{Benaych-Georges2011a} (see also \cite{Baik2005}) to the case $L > 1$.
It shows in particular that the largest eigenvalues of $\hat{\mathbf{R}}$ converge to some limits depending directly of the eigenvalues of $\mathbf{\Gamma}$, provided that for all $k \in \{1, \ldots, KL\}$ the ratios $\frac{\gamma_{k}}{\sigma^2}$ are above $\sqrt{c}$. The threshold $\sqrt{c}$ can be interpreted as a minimal SNR above which the $k^{th}$ largest \textit{signal related} eigenvalues of $\hat{\mathbf{R}}$ splits from the largest \textit{noise related} eigenvalue $\lambda_{KL+1}(\hat{\mathbf{R}})$.

The next result shows, under hypothesis $\mathcal{H}_0$, a joint \textit{Central Limit Theorem} (CLT) on the largest eigenvalues of $\hat{\mathbf{R}}_1, \dots, \hat{\mathbf{R}}_L, \hat{\mathbf{R}}$. 
\begin{theorem}
    \label{theorem:CLT_lambda_R_Hat}
    Let Assumptions \ref{assumption:Regime_Grandes_Dimensions}-\ref{assumption:Convergence_gamma}
    hold. Assume moreover that $\Gammabs_1 = \ldots = \Gammabs_L$ (thus $\gamma_{k,\ell} = \gamma_k$) and that
    \begin{align}
        \gamma_{1} > \ldots > \gamma_K > \sigma^2 \max\{\sqrt{c},\sqrt{c}_1,\ldots,\sqrt{c}_L\}.
    \end{align}
    Then we have
    \begin{align}
        \sqrt{M}&
        \begin{pmatrix}
            \lambda_k\left(\hat{\mathbf{R}}\right) - \phi_{c}(\gamma_k,\sigma^2)
            \\
            \left(\lambda_k\left(\hat{\mathbf{R}}_\ell\right) - \phi_{c_{\ell}}(\gamma_{k},\sigma^2)\right)_{\ell=1,\ldots,L}
        \end{pmatrix}_{\substack{k = 1, \dots, K}}
        \notag\\ 
        & \quad \xrightarrow[M\to\infty]{\mathcal{D}} \mathcal{N}_{\mathbb{R}^{K(L+1)}}\left(\mathbf{0},\mathbf{\Theta}\right),
    \end{align}
    where $\Thetabs$ is a positive definite block diagonal matrix given by $\boldsymbol{\Theta} = \mathrm{bdiag}\left(\mathbf{\Theta}_1, \dots, \mathbf{\Theta}_K\right)$ with
    \begin{equation}
        \boldsymbol{\Theta}_k \vcentcolon=
        \begin{pmatrix}
            \theta_{k, 0}^2 & \vartheta_{k, 1} & \ldots & \vartheta_{k, L}
            \\
            \vartheta_{k, 1} & \ddots & (0)
            \\
            \vdots & (0) & \ddots
            \\
            \vartheta_{k, L} & & & \theta_{k, L}^2
        \end{pmatrix},
    \end{equation}
    and by denoting $c_0 = c$,
    \begin{equation}
        \theta_{k, \ell}^2 = c_{\ell} \frac{(\gamma_k^2-\sigma^4 c_{\ell})(\gamma_k+\sigma^2)^2}{\gamma_k^2}, \quad \ell \geq 0,
    \end{equation}
    \begin{equation}
        \vartheta_{k, \ell} = c_{0} \frac{(\gamma_k^2-\sigma^4 c_{\ell})(\gamma_k+\sigma^2)^2}{\gamma_k^2}, \quad \ell \geq 1.
    \end{equation}
\end{theorem}
\begin{proof}
    The proof is postponed to Appendix \ref{section:Proof_Theorem_CLT_lambda_R_Hat}.
\end{proof} 
The result of Theorem \ref{theorem:CLT_lambda_R_Hat} provides an extension of \cite[Th. 1.4]{Benaych-Georges2011} which studies a CLT for the $K$ largest eigenvalues of $\hat{\R}_{\ell}$.
We note that the result of Theorem \ref{theorem:CLT_lambda_R_Hat} cannot be inferred directly from \cite[Th. 1.4]{Benaych-Georges2011} and requires a careful study due to the strong dependency between the eigenvalues of $\hat{\R}$ and the ones of $\hat{\R}_{1},\ldots,\hat{\R}_L$.

Theorems \ref{theorem:Spike_model_Limits} and \ref{theorem:CLT_lambda_R_Hat} are exploited in the following section to build a new statistic for the test \eqref{eq:General_Test}, relevant for the low-rank model \eqref{eq:Low_Rank_Model}.

\section{Proposed test statistic}
\label{section:Test_Statistic}

\par Before introducing our new test statistic, we first notice that the test \eqref{eq:General_Test} can be reformulated as:
\begin{align}
    \label{eq:Test_Low_Rank_1}
    \begin{split}
      \mathcal{H}_0 :& \quad \boldsymbol{\Gamma}_1 = \ldots = \boldsymbol{\Gamma}_L\\
      \mathcal{H}_1 :& \quad \exists (i,j) \in \{1, \dots, L\}^2 \text{ s.t. } \Gammabs_i \neq \Gammabs_{j}
    \end{split},
\end{align}
and we assume in the following that the rank $K$ is \emph{known} (see also Remark \ref{remark:unknown_rank} below in case the rank is unknown).
Next, we consider the following lemma, which shows that hypothesis $\Hcal_0$ can be verified by comparing the eigenvalues of matrix $\Gammabs$ with the ones of matrices $\Gammabs_1,\ldots,\Gammabs_L$.
\begin{lemma}
    \label{lemma:Valeurs_Propores_Gamma}
    The following assertions are equivalent:
    % % \vspace{0.15cm}
    \begin{itemize}
        \item[(a)] $\boldsymbol{\Gamma}_1 = \ldots = \boldsymbol{\Gamma}_L$,
        % % \vspace{0.15cm}
        \item[(b)] For all $k=1,\ldots,K$, $\ell = 1,\ldots,L$, $\lambda_k\left(\boldsymbol{\Gamma}_{\ell}\right) = \lambda_k\left(\boldsymbol{\Gamma}\right)$. 
    \end{itemize}
\end{lemma}
\begin{proof}
    % The proof of Lemma \ref{lemma:Valeurs_Propores_Gamma} is detailed in Appendix \ref{section:Proof_Valeurs_Propores_Gamma}.
    The proof of Lemma \ref{lemma:Valeurs_Propores_Gamma} can be found in \cite{Beisson2021SSP}.
\end{proof} 
From Lemma \ref{lemma:Valeurs_Propores_Gamma}, one can equivalently formulate the test \eqref{eq:Test_Low_Rank_1} as follows
\begin{align}
    \label{eq:Test_Low_Rank_2}
    \begin{split}
        \mathcal{H}_0 :& \quad \forall k, \ell, \  \lambda_{k}\left(\boldsymbol{\Gamma}_{\ell}\right) = \lambda_k\left(\boldsymbol{\Gamma}\right) \\
        \mathcal{H}_1 :& \quad \exists k,\ell :  \lambda_{k}\left(\boldsymbol{\Gamma}_{\ell}\right) \neq \lambda_k\left(\boldsymbol{\Gamma}\right)
    \end{split}.
\end{align}
Consequently, it is possible to discriminate between hypotheses $\mathcal{H}_0$ and $\mathcal{H}_1$ by exploiting only the eigenvalues of the matrices $\boldsymbol{\Gamma}_1, \dots,  \boldsymbol{\Gamma}_L, \boldsymbol{\Gamma}$ for which we can also build consistent estimators in the high-dimensional regime as follows.
Let us consider first the maximum likelihood estimator of the noise variance $\sigma^2$ given by
\begin{equation}
    \label{eq:Estimateur_Sigma_2}
    \hat{\sigma}^2 \vcentcolon= \sum_{\ell=1}^L \frac{N_{\ell}}{N} \frac{1}{M-K} \sum_{k=K+1}^M \lambda_k\left(\hat{\mathbf{R}}_{\ell}\right).
\end{equation}
From \eqref{eq:Low_Rank_Model} and Theorem \ref{theorem:Spike_model_Limits}, one can easily show that $\hat{\sigma}^2 \xrightarrow[]{} \sigma^2$ a.s. as $M\to+\infty$ under both $\Hcal_0$ and $\Hcal_1$. Next, for all $k \in\{1, \dots, KL\}$, let $\hat{\gamma}_k$ be the largest solution to the equation $\phi_c(\gamma_k,\hat{\sigma}^2) = \lambda_k(\hat{\R})$ if $\lambda_k(\hat{\R})>\hat{\sigma}^2(1+\sqrt{c})^2$, or $\hat{\gamma}_k = \hat{\sigma}^2\sqrt{c}$ otherwise. Similarly, for all $k \in \{1,\ldots,K\}$, let $\hat{\gamma}_{k, \ell}$ be the largest solution to the equation $\phi_{c_\ell}(\gamma_{k, \ell},\hat{\sigma}^2) = \lambda_k(\hat{\R}_{\ell})$ if $ \lambda_k(\hat{\R}_\ell)>\hat{\sigma}^2(1+\sqrt{c}_\ell)^2$, or $\hat{\gamma}_{k, \ell} = \hat{\sigma}^2\sqrt{c_\ell}$ otherwise. Then we have the following immediate result, as a consequence of Theorem \ref{theorem:Spike_model_Limits}.
\begin{corollary}
    \label{corollary:gamma_Estimate}
    Under Assumptions \ref{assumption:Regime_Grandes_Dimensions} and \ref{assumption:Convergence_gamma},
    \begin{equation}
        \hat{\gamma}_k \xrightarrow[M\to\infty]{a.s.}
        \begin{cases}
            \gamma_k & \text{ if } \gamma_k > \sigma^2 \sqrt{c}
            \\
            \sigma^2 \sqrt{c} & \text{ otherwise}
        \end{cases},
    \end{equation}
    \begin{equation}
        \hat{\gamma}_{k, \ell} \xrightarrow[M\to\infty]{a.s.}
        \begin{cases}
            \gamma_{k, \ell} & \text{ if } \gamma_{k, \ell} > \sigma^2 \sqrt{c_\ell}
            \\
            \sigma^2 \sqrt{c_\ell} & \text{ otherwise}
        \end{cases}.
    \end{equation}
\end{corollary}
Considering this result we propose the following test statistic
\begin{equation}
    \label{eq:Statistique_Test_T}
    T(\epsilon) = \mathbb{1}_{(\epsilon,+\infty)}\left(\left\|\hat{\boldsymbol{\gamma}} \right\|^2_2\right),
\end{equation}
where
\begin{equation}
\label{eq:gamma_vector}
    \hat{\boldsymbol{\gamma}} = \left(\hat{\gamma}_k - \hat{\gamma}_{k,\ell}\right)_{\substack{k = 1, \dots, K \\ \ell = 1, \dots, L}}.
\end{equation}
To study the performance in terms of consistency and asymptotic type I error of the test statistic 
\eqref{eq:Statistique_Test_T}, we consider the following assumption which 
ensures that the signal and noise eigenvalues of matrices $\hat{\R}, \hat{\R}_1,\ldots,\hat{\R}_L$ are separated in the \textit{high-dimensional regime}.
\begin{assumption}
    \label{assumption:Convergence_gamma_Separation_Condition}
    For all $k \in \{1, \dots, K\}$ and $\ell \in \{1, \dots, L\}$,
    \begin{align}
        \gamma_{1,\ell} > \ldots > \gamma_{K,\ell} &> \sigma^2 \max\{\sqrt{c}_1,\ldots,\sqrt{c_{L}}\},
        \\
        \gamma_1 > \ldots > \gamma_K &> \sigma^2\sqrt{c}.
    \end{align}
     Moreover, under $\mathcal{H}_1$, there exist $k, \ell$ such that $\gamma_k \neq \gamma_{k,\ell}$.
\end{assumption}
As a consequence of Corollary \ref{corollary:gamma_Estimate}, we have under Assumptions \ref{assumption:Regime_Grandes_Dimensions}-\ref{assumption:Convergence_gamma_Separation_Condition},
\begin{align}
    \left\|\hat{\boldsymbol{\gamma}}\right\|^2_2
    \xrightarrow[M\to\infty]{\mathrm{a.s.}} \left\|\boldsymbol{\gamma}\right\|^2_2,
    \label{eq:conv_gamma_diff}
\end{align}
with 
\begin{align}
    \boldsymbol{\gamma} = \left(\gamma_k - \gamma_{k,l}\right)_{\substack{k = 1, \dots, K \\ \ell = 1, \dots, L}},   
\end{align}
such that $\gammabs = \mathbf{0}$ under $\Hcal_0$ and $\gammabs \neq \mathbf{0}$ under $\Hcal_1$.
This implies the following consistency result.
\begin{theorem}
    \label{theorem:consistency_T}
  Let Assumptions \ref{assumption:Regime_Grandes_Dimensions}-\ref{assumption:Convergence_gamma_Separation_Condition} hold and denote $\epsilon_1 = \left\|\gammabs\right\|_2^2 > 0$ under $\Hcal_1$. Then for all $\epsilon \in \left(0,\epsilon_1\right)$,
  \begin{equation}
    \mathbb{P}_i\left(\lim_{M \to \infty} T(\epsilon) = i\right) = 1,
  \end{equation}
  for $i \in \{0,1\}$, where $\mathbb{P}_i$ is the probability measure under hypothesis $\mathcal{H}_i$.
\end{theorem}
To control the asymptotic type I error of the proposed test statistic \eqref{eq:Statistique_Test_T}, we also need the following result which, as a consequence of Theorem \ref{theorem:CLT_lambda_R_Hat}, provides a CLT for \eqref{eq:gamma_vector}.
% related to the estimators $\left(\hat{\gamma_k},\hat{\gamma}_{k,1},\ldots,\hat{\gamma}_{k,L}\right)_{k=1,\ldots,K}$, 
\begin{corollary}
\label{corollary:CLT_gamma}
     Under hypothesis $\mathcal{H}_0$ and Assumptions \ref{assumption:Regime_Grandes_Dimensions}-\ref{assumption:Convergence_gamma_Separation_Condition}, we have
    \begin{equation}
        \sqrt{M} \hat{\boldsymbol{\gamma}} \xrightarrow[M\to\infty]{\mathcal{D}} \mathcal{N}_{\mathbb{R}^{KL}}\left(\mathbf{0},\mathbf{H}\Upsilonbs\mathbf{H}^T \right),
    \end{equation}
    where $\H$ is the $KL \times K(L+1)$ matrix defined by $\mathbf{H} = \mathrm{bdiag}\left(\tilde{\mathbf{H}}, \dots, \tilde{\mathbf{H}}\right)$, 
    $\Upsilonbs = \mathrm{bdiag}\left(\Upsilonbs_1, \dots, \Upsilonbs_K\right)$ with $\tilde{\mathbf{H}}$ the $L \times (L+1)$ matrix given by   
    \begin{equation}
        \tilde{\mathbf{H}} = 
        \begin{pmatrix}
          1 & -1 & 0 & \ldots & \ldots & 0\\
          1 & 0 & -1 & \ddots &  & \vdots\\
          \vdots & \vdots & \ddots & \ddots & \ddots& \vdots\\
          \vdots & \vdots & & \ddots & \ddots & 0\\
          1 & 0 & \dots & \dots & 0 & -1
        \end{pmatrix},
    \end{equation}
    and with
    \begin{equation}
        \label{eq:Upsilon_k}
        \Upsilonbs_k =
        \begin{pmatrix}
            \omega^2_{k, 0} & \xi_k & \ldots & \xi_k
            \\
            \xi_k & \ddots & (0)
            \\
            \vdots & (0) & \ddots
            \\
            \xi_k & & & \omega_{k, L}^2 
        \end{pmatrix},
    \end{equation}
where
    \begin{equation}
        \begin{aligned}
            \omega_{k, \ell}^2 &=  \frac{c_{\ell}\gamma_k^2(\gamma_k + \sigma^2)^2}{\gamma_k^2-\sigma^4 c_{\ell}}, \quad \ell = 0,\ldots,L,
            \\
            \xi_k &=  \frac{c_{0}\gamma_k^2 (\gamma_k + \sigma^2)^2}{\gamma_k^2 - \sigma^4 c_{0}}.
        \end{aligned}
    \end{equation}
\end{corollary}
\begin{proof}
    The proof is deferred to Appendix \ref{section:Delta_Method}.
\end{proof}
From Corollary \ref{corollary:CLT_gamma}, we can adjust the threshold $\epsilon$ in \eqref{eq:Statistique_Test_T} to control the asymptotic type I error in the \textit{high-dimensional regime}, as described in the next result. Let us define $\hat{\Upsilonbs} = \mathrm{bdiag}\left(\hat{\Upsilonbs}_1,\ldots,\hat{\Upsilonbs}_K\right)$ with
\begin{align}
    \hat{\Upsilonbs}_k =
    \begin{pmatrix}
        \hat{\omega}^2_{k, 0} & \hat{\xi}_k & \ldots & \hat{\xi}_k
        \\
        \hat{\xi}_k & \ddots & (0)
        \\
        \vdots & (0) & \ddots
        \\
        \hat{\xi}_k & & & \hat{\omega}_{k, L}^2 
    \end{pmatrix},
\end{align}
where
\begin{equation}
    \begin{aligned}
        \hat{\omega}_{k, \ell}^2 &=  \frac{c_{\ell}\hat{\gamma}_k^2(\hat{\gamma}_k + \hat{\sigma}^2)^2}{\hat{\gamma}_k^2-\hat{\sigma}^4 c_{\ell}}, \quad \ell \geq 0\\
        \hat{\xi}_k &=  \frac{c_{0}\hat{\gamma}_k^2 (\hat{\gamma}_k + \hat{\sigma}^2)^2}{\hat{\gamma}_k^2 - \hat{\sigma}^4 c_{0}}.
    \end{aligned}
\end{equation}
From Corollary \ref{corollary:gamma_Estimate}, it is clear that $\hat{\Upsilonbs}\to \Upsilonbs$ a.s. as $M\to\infty$.
\begin{theorem}
    \label{theorem:Convergence_Loi}
    Let $\mathbf{x} \in \mathcal{N}_{\Rbb^{KL}}(\mathbf{0}, \mathbf{I})$ and $F(t,\boldsymbol{\Xi}) = \mathbb{P}\left(\mathbf{x}^T \boldsymbol{\Xi} \mathbf{x} \leq t\right)$, $\alpha \in (0,1)$ and set
    \begin{equation}
        \hat{\epsilon} = \frac{1}{M} \inf\left\{t \in \mathbb{R}: F\left(t, \mathbf{H}\hat{\Upsilonbs}\mathbf{H}^T\right) \geq 1 - \alpha\right\}.
    \end{equation}
    Then under Assumptions \ref{assumption:Regime_Grandes_Dimensions} and \ref{assumption:Convergence_gamma_Separation_Condition}, we have
    \begin{equation}  
        \mathbb{P}_0(T(\hat{\epsilon})=1) \xrightarrow[M\to\infty]{} \alpha.
    \end{equation}
\end{theorem}
In practice, Theorem \ref{theorem:Convergence_Loi} is used as follows. For a fixed realization of $\hat{\Upsilonbs}$, we sample the distribution of the Gaussian quadratic form $\x^T \H \hat{\Upsilonbs} \H^T \x$ and the threshold $\hat{\epsilon}$ is then set as the $(1-\alpha)$-quantile of $\x^T \H \hat{\Upsilonbs} \H^T \x$.
\begin{remark}
    \label{remark:unknown_rank}
    For a more general approach where each $\Gammabs_{\ell}$ has unknown rank $K_{\ell}$, one can obtain consistent estimates of $K_1,\ldots,K_L$ thanks to Theorem \ref{theorem:Spike_model_Limits}.
    Assuming $K_1,\ldots,K_L$ fixed with respect to $M$, and if for $\ell\in \{1,\ldots,L\}$, $\gamma_{K_{\ell},\ell} > \sigma^2\sqrt{c}_{\ell}$, under Assumption \ref{assumption:Convergence_gamma} the quantity
    \begin{equation}
        \hat{K}_{\ell} = \max\left\{k : \lambda_k\left(\hat{\R}_{\ell}\right) > \sigma^2(1+\sqrt{c}_{\ell})^2 + \epsilon_{\ell}\right\},
    \end{equation}
    converges almost surely to $K_{\ell}$ in the \textit{high-dimensional regime}, for all $\epsilon_{\ell} \in \left(0,\phi_{c_{\ell}}\left(\gamma_{K_{\ell},\ell},\sigma^2\right) - \sigma^2(1+\sqrt{c_{\ell}})^2\right)$. This shows that we can build consistent test statistics to capture changes in the rank (see further \cite{Vallet2012b}).
\end{remark}
\begin{remark}
 It is easy to show that under Assumption \ref{assumption:Convergence_gamma_Separation_Condition}, the matrix $\H \Upsilonbs\H^T$ is non singular. Therefore, an alternative approach to obtain a test statistic with controlled asymptotic type I error would be to consider the statistic
 \begin{align}
     \tilde{T}(\epsilon) = \mathbb{1}_{(\epsilon,+\infty)}\left(M\left\| \left(\mathbf{H} \hat{\Upsilonbs} \mathbf{H}^T \right)^{-\frac{1}{2}} \hat{\boldsymbol{\gamma}} \right\|^2_2\right),
 \end{align}
 since from Corollary \ref{corollary:CLT_gamma}, we have
 \begin{align}
    M\left\| \left(\mathbf{H} \hat{\Upsilonbs} \mathbf{H}^T \right)^{-\frac{1}{2}} \hat{\boldsymbol{\gamma}} \right\|^2_2 \xrightarrow[M\to \infty]{\Dcal} \chi^2(KL).
 \end{align}
 Nevertheless, although this approach looks simpler, it appears that the covariance matrix $\H \Upsilonbs\H^T$ is ill-conditioned. 
This can be readily seen, e.g. in the special case where $c_1 = \ldots = c_L$ and for a large SNR. If $\kappa(\Upsilonbs_k)$ denotes the condition number of $\Upsilonbs_k$ defined in \eqref{eq:Upsilon_k}, then we can verify (details are omitted) that $\kappa(\Upsilonbs_k)$ scales with $\gamma^2$ as $\gamma^2\to\infty$.
Therefore, in practice, setting the threshold $\epsilon$ based on the $\chi^2(KL)$ distribution gives poor performance. 
\end{remark}

\section{Some comparisons with alternative methods}
\label{section:alt_methods}

In this section, we compare the test statistic given in \eqref{eq:Statistique_Test_T} with two relevant alternatives for the low-rank model \eqref{eq:Low_Rank_Model} and the \textit{high-dimensional regime} described in Assumption \ref{assumption:Regime_Grandes_Dimensions}. To that purpose, we consider scenarios involving a change of subspace/eigenvalues for the rank $K=1$ model $\Gammabs_{\ell} = \gamma_{1,\ell} \u_\ell \u_\ell^*$, where $\|\u_{\ell}\|_2=1$, and where $\gamma_{1,\ell}$ is independent of $M$.
We precise that our objective is not to provide an exhaustive analysis of all the possible scenarios under $\Hcal_1$, but to draw some performance comparisons, in terms of consistency, out of a few simple cases. In the remainder of this section, we also assume that $L=2$ and $N_1=N_2$ so that $c_1=c_2 =2 c$.

\subsection{A test based on spiked Fisher matrices}

Although test statistics of the form \eqref{eq:Linear_Spectral_Statistic_F_Matrix} are not consistent in the \textit{high-dimensional regime}, we can build consistent test statistics by exploiting the behaviour of the largest and smallest eigenvalues of the Fisher matrices $\hat{\R}_{2}^{-1}\hat{\R}_{1}$ (see \cite{Wang2017}). We propose
\footnote
{
Although outside the scope of this paper, we note that the results of \cite[Th. 6.1]{Wang2017} could be exploited to build a test statistic with controlled asymptotic type I error, which is not the case for \eqref{eq:test_fisher}.
}
to use  $T_{\mathrm{Fisher}}(\epsilon) = \mathbb{1}_{(\epsilon,+\infty)} (F)$ with
\begin{equation}
    \begin{aligned}
        F &=
        \sum_{\substack{\ell,\ell'=1 \\ \ell' \neq \ell}}^L\sum_{k=1}^K
        \Biggl[
        \left(\lambda_k\left(\hat{\mathbf{R}}_{\ell}^{-1}\hat{\mathbf{R}}_{\ell'}\right) - \nu_{\ell,\ell'}^+\right)^+
        \\
        &\qquad\qquad\qquad+
        \left(\nu^-_{\ell,\ell'} - \lambda_{M-k}\left(\hat{\mathbf{R}}_{\ell}^{-1}\hat{\mathbf{R}}_{\ell'}\right)\right)^+
          \Biggr],
    \end{aligned}
     \label{eq:test_fisher}
\end{equation}
 where $\nu_{\ell,\ell'}^{\pm} = \left(\frac{1 \pm \sqrt{c_{\ell}+c_{\ell'}-c_{\ell} c_{ \ell'}}}{1-c_{\ell}}\right)^2$.

\textit{Change of subspace.} Let us consider that under $\Hcal_1$, $\gamma_{1,1} = \gamma_{1,2}$ and $\u_1^*\u_2 \to 0$ as $M\to \infty$, so that the changes between $\Gammabs_1$ and $\Gammabs_2$ are only carried by the unit norm eigenvector $\u_2$. It is easily seen that
\begin{align}
  \lambda_k\left(\R_{2}^{-1} \R_{1}\right) \xrightarrow[M \to \infty]{}
  \begin{cases}
    \frac{\gamma_{1,1}+\sigma^2}{\sigma^2} & \text{ if } k=1,
    \\
    1 & \text{ if } k=2,\ldots,M-1,
    \\
    \frac{\sigma^2}{\gamma_{1,1}+\sigma^2} &  \text{ if } k=M,
  \end{cases}
\end{align}
so that applying \cite[Th. 3.1]{Wang2017} shows that for all small $\epsilon > 0$, $\Pbb_i(\lim T_{\mathrm{Fisher}}(\epsilon) = i) = 1$ for $i \in \{0,1\}$ as $M\to\infty$ iff
\begin{align}
  \frac{\gamma_{1,1}}{\sigma^2} > \beta = \frac{2 \left(c +\sqrt{c-c^2}\right)}{1-2c}.
\end{align}
One can see that $\beta > \sqrt{2c}$ and therefore, from Assumption \ref{assumption:Convergence_gamma_Separation_Condition} and Theorem \ref{theorem:consistency_T}, we deduce that the Fisher based statistic requires a larger SNR $\frac{\gamma_{1,1}}{\sigma^2}$ compared to the Wishart based statistic proposed in \eqref{eq:Statistique_Test_T} to be consistent in the change of subspace scenario.

\textit{Change of eigenvalues.} In that case, we assume that $\gamma_{1,2} = \gamma_{1,1}(1+\delta)$ with $\delta > 0$ and $\u_1 = \u_2$, so that the changes are only carried  by the largest eigenvalue of $\Gammabs_2$. Note that under these settings, Assumption \ref{assumption:Convergence_gamma} is verified and it holds that
\begin{align}
  \lambda_k\left(\R_{2}^{-1} \R_{1}\right) \xrightarrow[M \to \infty]{}
  \begin{cases}
    \frac{\gamma_{1,1}+\sigma^2}{\gamma_{1,1}(1+\delta) +\sigma^2} & \text{ if } k=1,
    \\
    1 & \text{ if } k=2,\ldots,M.
  \end{cases}
\end{align}
Using again \cite[Th. 3.1]{Wang2017}, we have that for all small $\epsilon > 0$, $\Pbb(\lim T_{\mathrm{Fisher}}(\epsilon) = i) = 1$ for $i \in {0,1}$ as $M \to \infty$ iff
\begin{align}
    \delta > \frac{2(c + \sqrt{c-c^2})}{1-2c},
    \label{eq:cond_delta_1}
\end{align}
and
\begin{align}
    \frac{\gamma_{1,1}}{\sigma^2} > \beta =\frac{2(c+\sqrt{c-c^2})}{(1+\delta)(1-2c) - (1 + 2\sqrt{c-c^2})}.
    \label{eq:cond_delta_2}
\end{align}
In this scenario, one can see that the minimal SNR  $\beta$ decreases when $\delta$ increases, which can be exploited to produce conditions where the Fisher test statistic is consistent while the Wishart one is not. Indeed, choose $\sqrt{c} < \frac{\gamma_{1,1}}{\sigma^2} < \sqrt{2c}$ and $\delta$ large enough so that \eqref{eq:cond_delta_1} and \eqref{eq:cond_delta_2} are verified. Then it can be seen from Corollary \ref{corollary:gamma_Estimate} that $\|\hat{\gammabs}\|_2^2 \to 2(\gamma_{1,1} - \sigma^2 \sqrt{2c})^2$ a.s. as $M\to \infty$ and therefore for all small $\epsilon > 0$, $\Pbb_0\left(\lim T(\epsilon) = 1\right) =1$.

\subsection{The GLR for \eqref{eq:Test_Low_Rank_1}}

As an alternative to the GLR for the general covariance equality test \eqref{eq:General_Test}, the GLR for the low-rank test \eqref{eq:Test_Low_Rank_1} can be derived. Classical computations (details are omitted) provide the following test statistic $T_{\mathrm{GLR -LR}}(\epsilon) = \mathbb{1}_{(\epsilon,+\infty)}(G)$ where
\begin{equation}
    \begin{aligned}
        &G = - \sum_{\ell=1}^L N_\ell \sum_{k = 1}^K \log\left(\frac{\lambda_k(\hat{\textbf{R}}_\ell)}{\lambda_k(\hat{\textbf{R}})}\right)\\
        &- N (M - K) \log\left( \frac{\frac{1}{M - K} \sum\limits_{\ell=1}^L \frac{N_\ell}{N}\sum\limits_{k = K+1}^M \lambda_k(\hat{\textbf{R}}_\ell)}{\frac{1}{M - K}\sum\limits_{k = K+1}^M \lambda_k(\hat{\textbf{R}})} \right).
    \end{aligned}
    \label{eq:G_Expression}
\end{equation}
Using Theorem \ref{theorem:Spike_model_Limits}, it can be shown that $G \to G_{\infty}$ a.s. as $M\to\infty$
where
\begin{align}
    G_{\infty} =
    \sum_{\ell=1}^L \frac{c}{c_{\ell}} \sum_{k=1}^K
    \left(\psi\left(\frac{\phi_{c}(\gamma_k)}{\sigma^2}\right) - \psi\left(\frac{\phi_{c_{\ell}}(\gamma_{k,\ell})}{\sigma^2}\right)\right),
    \label{eq:GLR_limit}
\end{align}
with $\psi(x) = x - \log(x)$. 

Let us consider a change of eigenvalues with $\gamma_{1,2} = \gamma_{1,1} + \delta$ and $\u_1 = \u_2$ under $\Hcal_1$. Then it is easy to see that under both $\Hcal_0$ and $\Hcal_1$, $G_{\infty} = - c +\Ocal\left(\frac{1}{\gamma_{1,1}}\right)$ as $\gamma_{1,1} \to +\infty$. Regarding the proposed test \eqref{eq:Statistique_Test_T}, we have 
$\left\|\gammabs\right\|_2^2 = \frac{\delta^2}{2}$ under $\Hcal_1$ which shows the limit \eqref{eq:conv_gamma_diff} under $\Hcal_0$ and $\Hcal_1$ cannot be made arbitrarily close as $\gamma_{1,1} \to \infty$. This suggests that for a large $\gamma_{1,1}$ and a fixed change $\delta$, the GLR for the low-rank model might experience a performance loss compared to the test \eqref{eq:Statistique_Test_T}.

\section{Simulations}
\label{section:Simulations}

\par In this section, we provide simulations to illustrate the performance of the test statistic $T$ proposed in \eqref{eq:Statistique_Test_T}, and to perform some comparisons with the alternative test statistics introduced in Section \ref{section:alt_methods}. 
We consider $\sigma^2 = 0.5$, $K = 2, L = 2$ as well as a Toeplitz model of rank $K=2$ for the covariance matrix $\Gammabs_{\ell}$ which can therefore be written as
\begin{align}
        \boldsymbol{\Gamma}_\ell = \gamma_{1, \ell} \textbf{a}\left(\theta_{1,\ell} \right)\textbf{a}^*\left(\theta_{1, \ell}\right)
        + \gamma_{2, \ell}  \textbf{a}\left(\theta_{2, \ell}\right)
        \textbf{a}^*\left(\theta_{2, \ell}\right),
    \label{eq:Gamma_Simulations}
\end{align}
with $\mathbf{a}(\theta) = \frac{1}{\sqrt{M}} (1, \mathrm{e}^{\mathrm{i} \theta}\ldots,\mathrm{e}^{\mathrm{i} (M-1) \theta})^T$. Note that the model \eqref{eq:Gamma_Simulations} is common in spectral analysis and array processing \cite{stoica_spectral_2005}.
    
    \subsection{Empirical and asymptotic Type I error of $T(\epsilon)$}
    
We first illustrate the result of Theorem \ref{theorem:Convergence_Loi} and consider $\theta_{k,\ell} = 0$ for $(k, \ell) \in \{1,2\}^2$, $\gamma_{1,\ell} = 3$ and $\gamma_{2,\ell} = 1.5$ for $\ell \in \{1,2\}$ and $N_1=N_2=2M$. The threshold $\epsilon$ of \eqref{eq:Statistique_Test_T} is set as the $(1-\alpha)$-quantile of the Gaussian quadratic form $\mathbf{x}^T\mathbf{H}\hat{\Upsilonbs}\mathbf{H}^T\mathbf{x}$ with $\mathbf{x} \sim \mathcal{N}_{\mathbb{R}^{KL}}(\mathbf{0}, \mathbf{I})$, and we provide in \textsc{Table} \ref{tab:ErrorTypeI} the empirical Type I error of $T(\epsilon)$ (evaluated over 100000 iterations) for $M\in \{10,20,50,100\}$.
\begin{table}[ht]
    \centering
    \caption{Type I error of $T(\epsilon)$}
    \begin{tabular} { |P{2cm}||P{0.8cm}|P{0.8cm}|P{0.8cm}|P{0.8cm}|P{0.8cm}| }
        \hline
        \diagbox[width=2.3cm]{$T(\epsilon)$}{$\alpha$} & 0.005 & 0.01 & 0.02 & 0.05 & 0.10\\
        \hline
        $M = 10$        & 0.002 & 0.004 & 0.009 & 0.028 & 0.065\\
        \hline
        $M = 20$        & 0.0025 & 0.005 & 0.01 & 0.03 & 0.073\\
        \hline
        $M = 50$        & 0.003 & 0.006 & 0.013 & 0.038 & 0.083\\
        \hline
        $M = 100$       & \textbf{0.004} & \textbf{0.008} & \textbf{0.016} & \textbf{0.043} & \textbf{0.09}\\
        \hline
    \end{tabular}
    \label{tab:ErrorTypeI}
\end{table}
Table \ref{tab:ErrorTypeI} thus shows that the empirical type I error is close to the asymptotic
type I error predicted in Theorem \ref{theorem:Convergence_Loi}, when $M$ is increasing.

    \subsection{Comparisons of powers}

In this section, we evaluate the proposed test statistic on synthetic data by considering the following two scenarios.
\begin{enumerate}
    \item[(1)] \textit{Change of subspace:} under $\mathcal{H}_0$, $\theta_{1, 1} = \theta_{1, 2} = 0$, $\theta_{2, 1} = \theta_{2, 2} = \frac{\pi}{8}$ and under $\mathcal{H}_1$,  $\theta_{1, 1} = 0$, $\theta_{1, 2} = \frac{\pi}{2}$, $\theta_{2, 1} = \frac{\pi}{8}$, $\theta_{2, 2} = \frac{\pi}{2} + \frac{\pi}{8}$. We will also consider under both hypothesis $\gamma_{1, 1} = \gamma_{1, 2} = 2$, $\gamma_{2, 1} = \gamma_{2, 2} = 1$ and  $N_1 = N_2 = 2M$ thus $c_1 = c_2$.
    \item[(2)] \textit{Change of Eigenvalues:} under $\mathcal{H}_0$, $\gamma_{1, 1} = \gamma_{1, 2} = 2$, $\gamma_{2, 1} = \gamma_{2, 2} = 1.5$ and under $\mathcal{H}_1$, $\gamma_{1, 1} = 2$, $\gamma_{1, 2} = 5$,  $\gamma_{2, 1} = 1.5$, $\gamma_{2, 2} = 4$.  We will also consider under both hypothesis $\theta_{1, 1} = \theta_{1, 2} = 0$, $\theta_{2, 1} = \theta_{2, 2} = 0$ and $N_1 = N_2 = 4M$ thus $c_1 = c_2$.
\end{enumerate}
In the simulations that follow, for both scenarios,  we compute the power of different test statistics for a given type I error $\alpha$ and different values of $M$. The statistic $T(\epsilon)$, which will be termed as "Wishart" below, will be compared to the statistics $T_{\mathrm{Fisher}}(\epsilon)$ (termed as "Fisher"), $T_{GLR-LR}(\epsilon)$ (termed as "GLR-LR") and $T_{GLR}(\epsilon) = \mathbb{1}_{(\epsilon,+\infty)} (S_{|_{\varphi = \log}})$ where $S$ is given in \eqref{eq:Linear_Spectral_Statistic_F_Matrix} (termed as "GLR"). For each of the statistics, the threshold $\epsilon$ is adjusted separately to achieve a type I error of $\alpha$. Note that for both scenarios, Assumptions \ref{assumption:Convergence_gamma}, \ref{assumption:Convergence_gamma_Separation_Condition} are verified and that the condition for the Fisher statistic to be consistent is verified as well.
\begin{table}[ht]
    \centering
    \caption{Power for different values of $M$ (change of eigenvalues scenario)}
    \begin{tabular}{ |P{2.5cm}||P{0.7cm}|P{0.7cm}|P{0.7cm}|P{0.7cm}|P{0.7cm}|  }
        \hline
        \diagbox[width=2.4cm]{Statistics}{$\alpha$} & 0.005 & 0.01 & 0.02 & 0.05 & 0.1 \\
        \hline
        \multicolumn{6}{|c|}{$M = 10$} \\
        \hline
        GLR                 & 0.120 & 0.181 & 0.266 & 0.412 & 0.550\\
        \hline
        GLR-LR              & 0.381 & 0.483 & 0.588 & 0.734 & 0.832\\
        \hline
        Fisher              & 0.309 & 0.397 & 0.5 & 0.653 & 0.775\\
        \hline
        Wishart             & \textbf{0.998} & \textbf{0.999} & \textbf{0.999} & \textbf{1} & \textbf{1}\\
        \hline
        \multicolumn{6}{|c|}{$M = 20$} \\
        \hline
        GLR                 & 0.137 & 0.197 & 0.277 & 0.424 & 0.569\\
        \hline
        GLR-LR              & 0.736 & 0.808 & 0.87 & 0.934 & 0.967\\
        \hline
        Fisher              & 0.578 & 0.672 & 0.762 & 0.861 & 0.923\\
        \hline
        Wishart             & \textbf{1} & \textbf{1} & \textbf{1} & \textbf{1} & \textbf{1}\\
        \hline
        \multicolumn{6}{|c|}{$M = 50$} \\
        \hline
        GLR                 & 0.145 & 0.209 & 0.297 & 0.445 & 0.591\\
        \hline
        GLR-LR              & 0.992 & 0.996 & \textbf{1} & \textbf{1} & \textbf{1}\\
        \hline
        Fisher              & 0.946 & 0.965 & 0.98 & 0.992 & 0.997\\
        \hline
        Wishart             & \textbf{1} & \textbf{1} & \textbf{1} & \textbf{1} & \textbf{1}\\
        \hline
        \multicolumn{6}{|c|}{$M = 100$} \\
        \hline
        GLR                 & 0.154 & 0.207 & 0.290 & 0.445 & 0.591\\
        \hline
        GLR-LR              & \textbf{1} & \textbf{1} & \textbf{1} & \textbf{1} & \textbf{1}\\
        \hline
        Fisher              & 0.998 & 0.999 & \textbf{1} & \textbf{1} & \textbf{1}\\
        \hline
        Wishart             & \textbf{1} & \textbf{1} & \textbf{1} & \textbf{1} & \textbf{1}\\
        \hline
    \end{tabular}
    \label{tab:EigenvaluesScenario_M}
\end{table}
\begin{table}[ht]
    \centering
    \caption{Power for different values of $M$ (change of subspace scenario)}
    \begin{tabular}{ |P{2.5cm}||P{0.7cm}|P{0.7cm}|P{0.7cm}|P{0.7cm}|P{0.7cm}|  }
        % \hline
        % & \multicolumn{5}{|c|}{PD Values} \\
        \hline
        \diagbox[width=2.4cm]{Statistics}{$\alpha$} & 0.005 & 0.01 & 0.02 & 0.05 & 0.1 \\
        \hline
        \multicolumn{6}{|c|}{$M = 10$} \\
        \hline
        GLR                 & 0.493 & 0.592 & 0.701 & 0.826 & 0.906\\
        \hline
        GLR-LR              & \textbf{0.992} & \textbf{0.996} & \textbf{0.998} & \textbf{0.999} & \textbf{1}\\
        \hline
        Fisher              & 0.149 & 0.215 & 0.312 & 0.473 & 0.624\\
        \hline
        Wishart             & 0.026 & 0.056 & 0.119 & 0.3 & 0.519\\
        \hline
        \multicolumn{6}{|c|}{$M = 20$} \\
        \hline
        GLR                 & 0.757 & 0.829 & 0.89 & 0.949 & 0.978\\
        \hline
        GLR-LR              & \textbf{1} & \textbf{1} & \textbf{1} & \textbf{1} & \textbf{1}\\
        \hline
        Fisher              & 0.398 & 0.493 & 0.597 & 0.739 & 0.84\\
        \hline
        Wishart             & 0.646 & 0.812 & 0.924 & 0.988 & \textbf{1}\\
        \hline
        \multicolumn{6}{|c|}{$M = 50$} \\
        \hline
        GLR                 & 0.832 & 0.883 & 0.927 & 0.968 & 0.987\\
        \hline
        GLR-LR              & \textbf{1} & \textbf{1} & \textbf{1} & \textbf{1} & \textbf{1}\\
        \hline
        Fisher              & 0.783 & 0.846 & 0.894 & 0.944 & 0.972\\
        \hline
        Wishart             & \textbf{1} & \textbf{1} & \textbf{1} & \textbf{1} & \textbf{1}\\
        \hline
        \multicolumn{6}{|c|}{$M = 100$} \\
        \hline
        GLR                 & 0.838 & 0.891 & 0.934 & 0.972 & 0.988\\
        \hline
        GLR-LR              & \textbf{1} & \textbf{1} & \textbf{1} & \textbf{1} & \textbf{1}\\
        \hline
        Fisher              & 0.955 & 0.971 & 0.984 & 0.993 & 0.997\\
        \hline
        Wishart             & \textbf{1} & \textbf{1} & \textbf{1} & \textbf{1} & \textbf{1}\\
        \hline
    \end{tabular}
    \label{tab:EigenvectorsScenario_M}
\end{table}
We observe that in both scenarios, the proposed Wishart statistic shows a significantly better performance as $M$ is increasing. Second, while the GLR-LR statistic outperforms the Wishart one for low dimensions $M$ in the change of subspace scenario, the Wishart statistic still demonstrates a higher power compared to the Fisher and GLR statistics for both scenarios.

The next simulation in Table \ref{tab:EigenVaScenario_sigma} shows the evolution of the power for different values of the noise variance in the change of eigenvalues scenario ($M = 100$).
\begin{table}[ht]
    \centering
    \caption{Power for different values of $\sigma^2$ (change of eigenvalues scenario)}
    \begin{tabular}{ |P{2.5cm}||P{0.7cm}|P{0.7cm}|P{0.7cm}|P{0.7cm}|P{0.7cm}|  }
        \hline
        \diagbox[width=2.4cm]{Statistics}{$\alpha$} & 0.005 & 0.01 & 0.02 & 0.05 & 0.1 \\
        \hline
        \multicolumn{6}{|c|}{$\sigma^2 = 0.75$} \\
        \hline
        GLR             & 0.1 & 0.153 & 0.226 & 0.368 & 0.512\\
        \hline
        GLR-LR          & 0.999 & \textbf{1} & \textbf{1} & \textbf{1} & \textbf{1}\\
        \hline
        Fisher          & 0.925 & 0.95 & 0.97 & 0.987 & 0.944\\
        \hline
        Wishart         & \textbf{1} & \textbf{1} & \textbf{1} & \textbf{1} & \textbf{1}\\
        \hline
        \multicolumn{6}{|c|}{$\sigma^2 = 1$} \\
        \hline
        GLR             & 0.079 & 0.121 & 0.185 & 0.310 & 0.448\\
        \hline
        GLR-LR          & 0.995 & 0.998 & 0.999 & \textbf{1} & \textbf{1}\\
        \hline
        Fisher          & 0.662 & 0.736 & 0.809 & 0.89 & 0.938\\
        \hline
        Wishart         & \textbf{1} & \textbf{1} & \textbf{1} & \textbf{1} & \textbf{1}\\
                \hline
        \multicolumn{6}{|c|}{$\sigma^2 = 5.5$} \\
        \hline
        GLR             & 0.008 & 0.019 & 0.037 & 0.083 & 0.152\\
        \hline
        GLR-LR          & 0.029 & 0.045 & 0.07 & 0.13 & 0.207\\
        \hline
        Fisher          & 0.008 & 0.015 & 0.029 & 0.07 & 0.133\\
        \hline
        Wishart         & \textbf{0.315} & \textbf{0.409} & \textbf{0.498} & \textbf{0.649} & \textbf{0.762}\\
        \hline
    \end{tabular}
    \label{tab:EigenVaScenario_sigma}
\end{table}
The same can be done for the change of subspace scenario in Table \ref{tab:EigenVecScenario_sigma}.
\begin{table}[ht]
    \centering
    \caption{Power for different values of $\sigma^2$ (change of subspace scenario)}
    \begin{tabular}{ |P{2.5cm}||P{0.7cm}|P{0.7cm}|P{0.7cm}|P{0.7cm}|P{0.7cm}|  }
        \hline
        \diagbox[width=2.4cm]{Statistics}{$\alpha$} & 0.005 & 0.01 & 0.02 & 0.05 & 0.1 \\
        \hline
        \multicolumn{6}{|c|}{$\sigma^2 = 0.75$} \\
        \hline
        GLR             & 0.4 & 0.496 & 0.602 & 0.749 & 0.849\\
        \hline
        GLR-LR          & \textbf{1} & \textbf{1} & \textbf{1} & \textbf{1} & \textbf{1}\\
        \hline
        Fisher          & 0.329 & 0.414 & 0.510 & 0.625 & 0.763\\
        \hline
        Wishart         & \textbf{1} & \textbf{1} & \textbf{1} & \textbf{1} & \textbf{1}\\
        \hline
        \multicolumn{6}{|c|}{$\sigma^2 = 1$} \\
        \hline
        GLR             & 0.172 & 0.246 & 0.34 & 0.505 & 0.646\\
        \hline
        GLR-LR          & \textbf{1} & \textbf{1} & \textbf{1} & \textbf{1} & \textbf{1}\\
        \hline
        Fisher          & 0.077 & 0.119 & 0.18 & 0.297 & 0.424\\
        \hline
        Wishart         & \textbf{1} & \textbf{1} & \textbf{1} & \textbf{1} & \textbf{1}\\
        \hline
        \multicolumn{6}{|c|}{$\sigma^2 = 2.5$} \\
        \hline
        GLR             & 0.017 & 0.031 & 0.057 & 0.120 & 0.209\\
        \hline
        GLR-LR          & \textbf{0.341} & \textbf{0.423} & \textbf{0.537} & \textbf{0.695} & \textbf{0.809}\\
        \hline
        Fisher          & 0.007 & 0.015 & 0.03 & 0.068 & 0.129\\
        \hline
        Wishart         & 0.256 & 0.327 & 0.414 & 0.566 & 0.655\\
        \hline
    \end{tabular}
    \label{tab:EigenVecScenario_sigma}
\end{table}
We observe that the test statistics designed for a low-rank scenario (Wishart, Fisher, GLR-LR) outperform the GLR in general. Additionally, when the noise variance $\sigma^2$ becomes too large, the conditions on the SNR ensuring the consistency of these statistics (Assumption \ref{assumption:Convergence_gamma_Separation_Condition} and the conditions of \cite{Wang2017}) are not met anymore, and one observes in that case a significant drop of the performance ($\sigma^2 = 5.5$ in Table \ref{tab:EigenVaScenario_sigma} and $\sigma^2 = 2.5$ in Table \ref{tab:EigenVecScenario_sigma}).

\begin{remark}
    In view of the simulations results described in Tables \ref{tab:EigenvaluesScenario_M}-\ref{tab:EigenVecScenario_sigma}, we observe that the proposed test statistic \eqref{eq:Statistique_Test_T} performs poorly when the conditions described in Assumptions 
    \ref{assumption:Regime_Grandes_Dimensions} and \ref{assumption:Convergence_gamma_Separation_Condition}
    are not met, i.e. when the dimension $M$ or the SNR are not large enough. Scenarios where the rank $K$
    is also of the same order of magnitude than $M$, thus violating Assumption \ref{assumption:Regime_Grandes_Dimensions}, will also invalidate Theorem \ref{theorem:Convergence_Loi}
    and the asymptotic type I error will be poorly controlled in that case.
\end{remark}
    
    \subsection{An application to change detection in SAR images}

In this section, we evaluate the performance of the proposed test statistic on images drawn from the UAV-SAR dataset of NASA/JPL-Caltech (SanAnd\_26524\_03, Segment 4). We consider two scenes with respective sizes $2360 \times 600$ and $2300 \times 600$ pixels, which have been previously used in \cite{Mian2018,Abdallah2019}, and which are formed  of $L=2$ images acquired within a 5 years interval (see Figures \ref{fig:scene1} and \ref{fig:scene2}). 
\begin{figure}
    \centering
    \includegraphics[width=8cm, height=5cm]{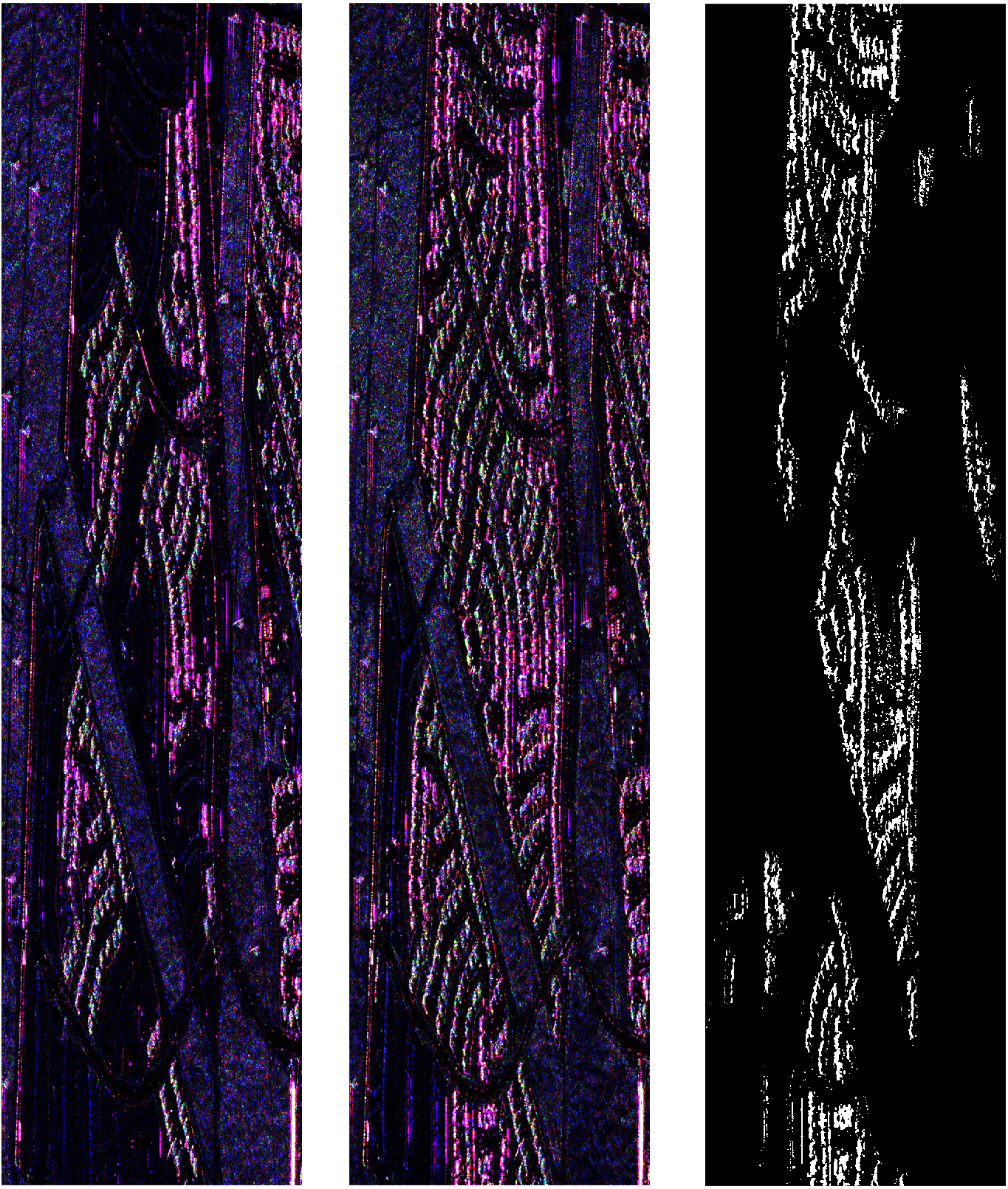}
    \caption{Scene 1 (Pauli representation) at two different times and its ground truth}
    \label{fig:scene1}
\end{figure}
\begin{figure}
    \centering
    \includegraphics[width=8cm, height=5cm]{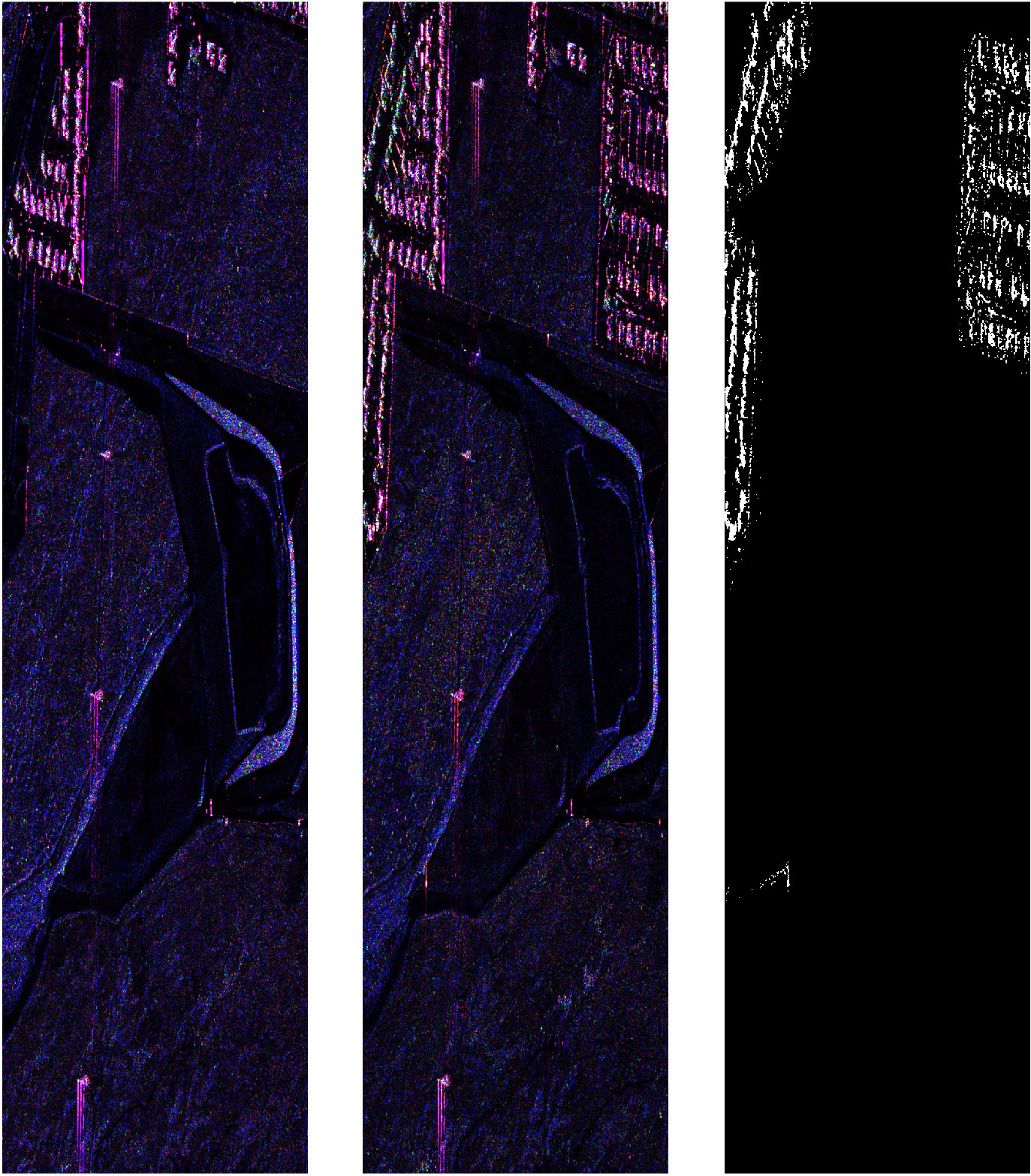}
    \caption{Scene 2 (Pauli representation) at two different times and its ground truth}
    \label{fig:scene2}
\end{figure}
The azimuthal resolution is approximately $0.6$ m while the distance resolution is $1.67$ m. The dimension of each pixel, which was initially of $M=3$, has been increased to $M=12$ using the wavelet decomposition technique of \cite{mian2019design}. Local patches of sizes $5 \times 5$ centered around each pixel under test are used for estimation, that is $N_1=N_2=25$. In Figure \ref{fig:cumsum}, the ratio
\begin{align}
    r(k)
    = 
    \frac
    {
         \Ebb\left[\sum_{i=1}^k\lambda_i\left(\hat{\R}_{1}\right)\right]
    }
    {
        \Ebb\left[\tr\left(\hat{\R}_{1}\right)\right]
    },
\end{align}
is plotted for both scenes, where the expectations are estimated by a sample mean over all the local patches. The rank $K$ is set to 5 in the following to reach a ratio of $r(K) \gtrapprox 95\%$.
\begin{figure}
     \centering
    \centerline
    {
        \subfigure[Scene 1]{\includegraphics[scale=0.07]{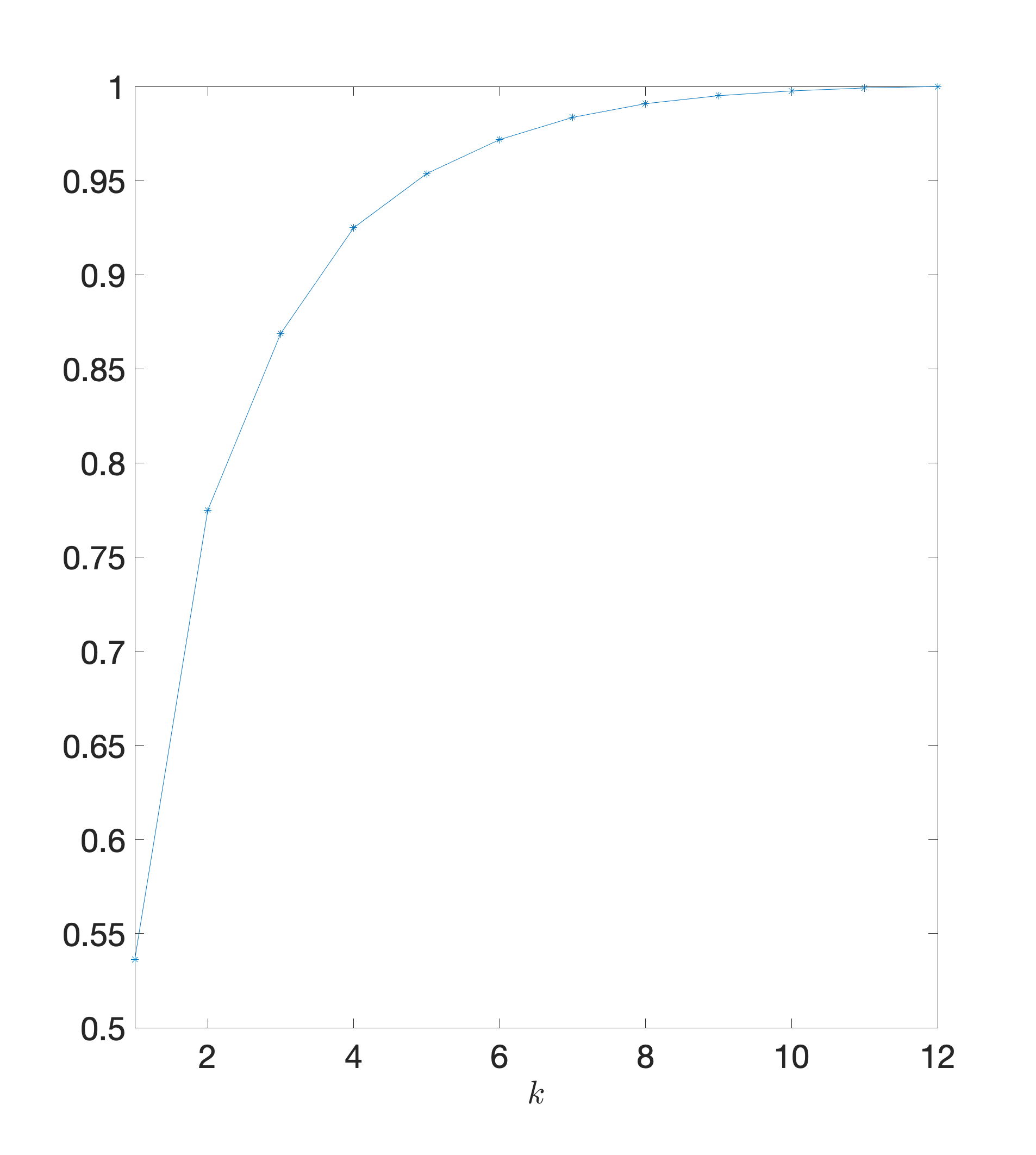}}
        \hfill
        \subfigure[Scene 2]{\includegraphics[scale=0.072]{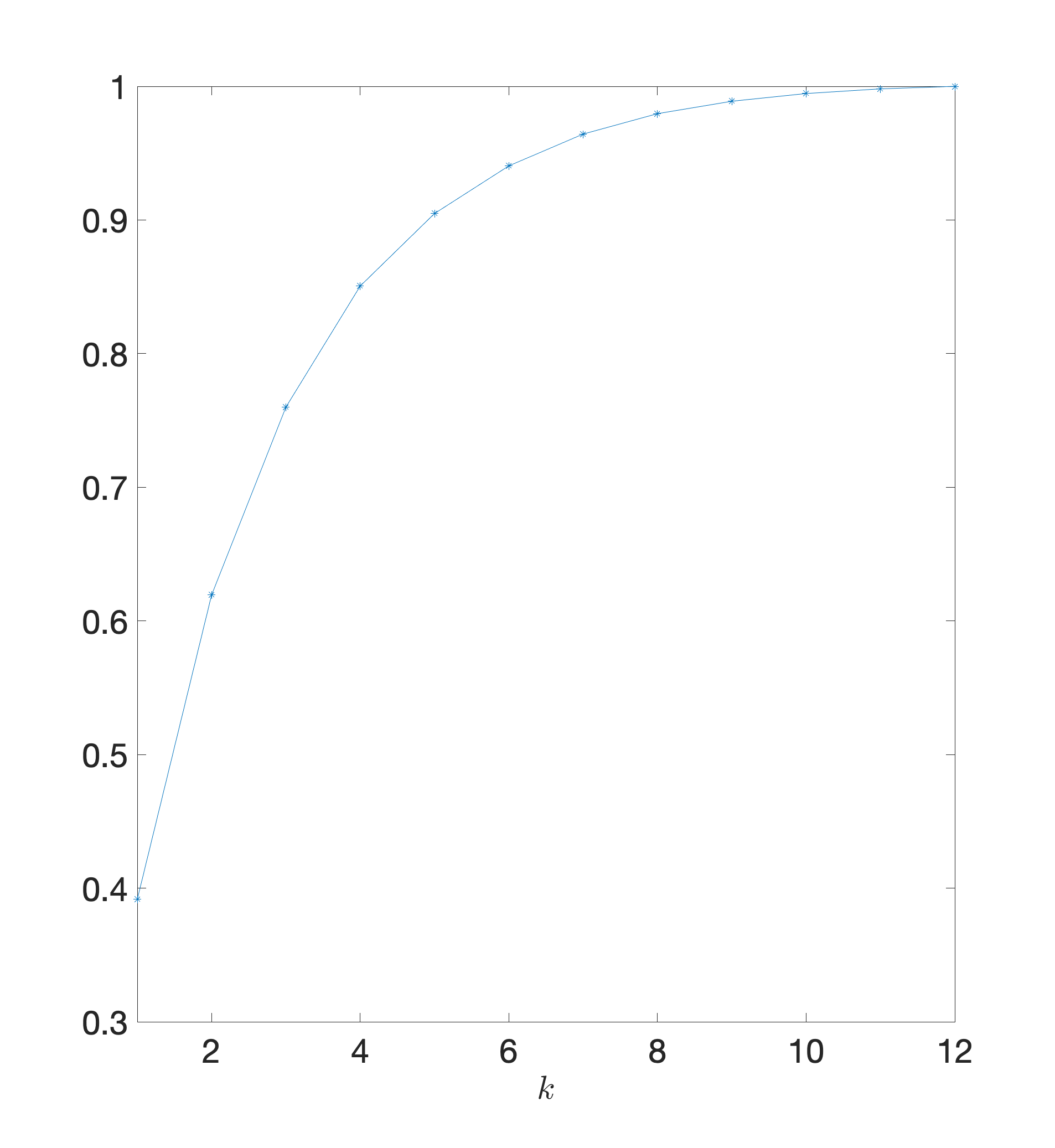}}
    }
    \caption{Ratio $k \mapsto r(k)$ for both scenes}
    \label{fig:cumsum}
\end{figure}
In Figure \ref{fig:SAR_ROC} are plotted the ROC curves for scenes 1 and 2, where we have compared the performance of the proposed test statistic \eqref{eq:Statistique_Test_T}, the GLR, the GLR-LR and the method of \cite{Abdallah2019}.
\begin{figure}
    \centering
    \centerline
    {
        \subfigure[Scene 1]{\includegraphics[scale=0.07]{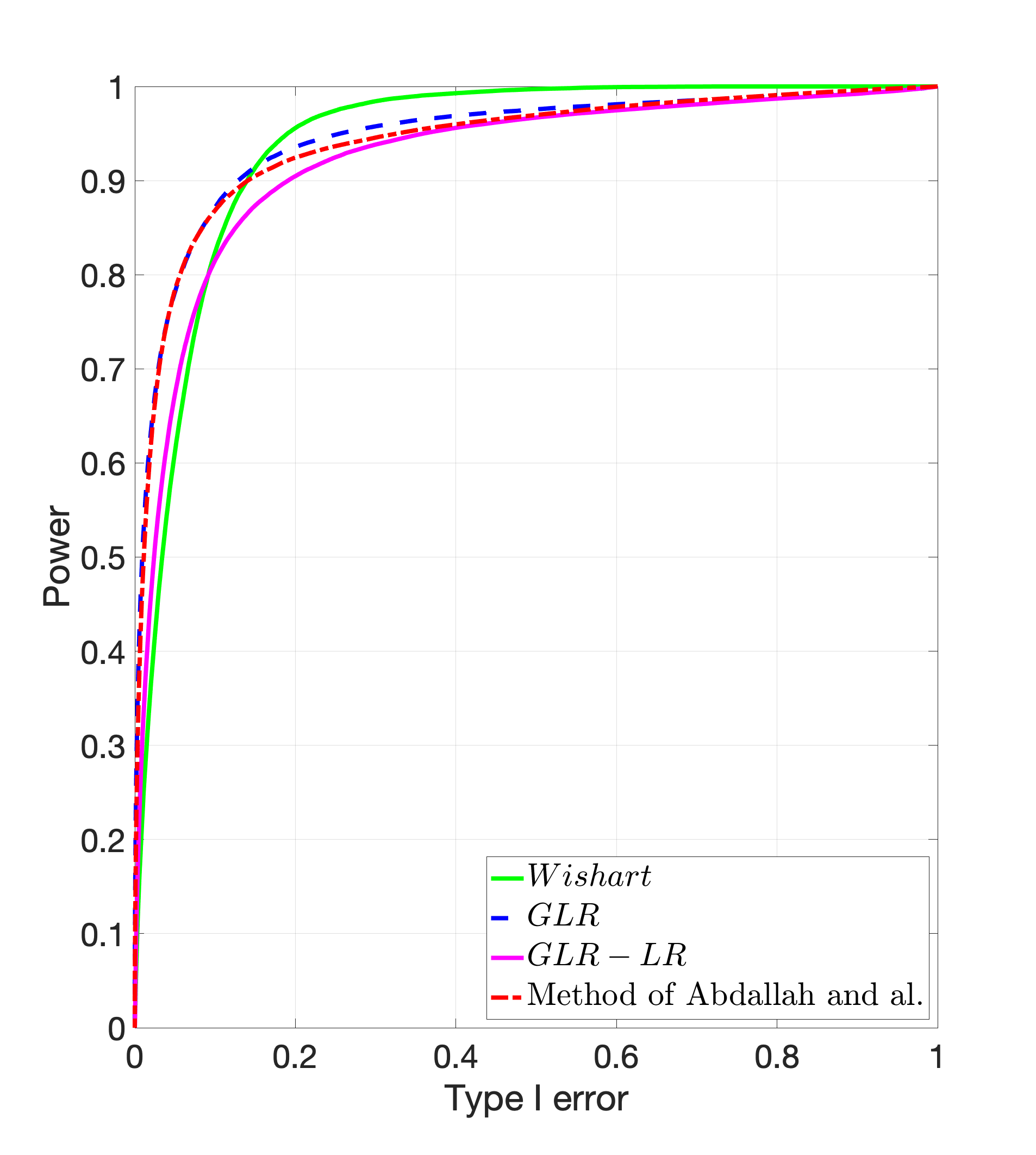}}
        \hfill
        \subfigure[Scene 2]{\includegraphics[scale=0.07]{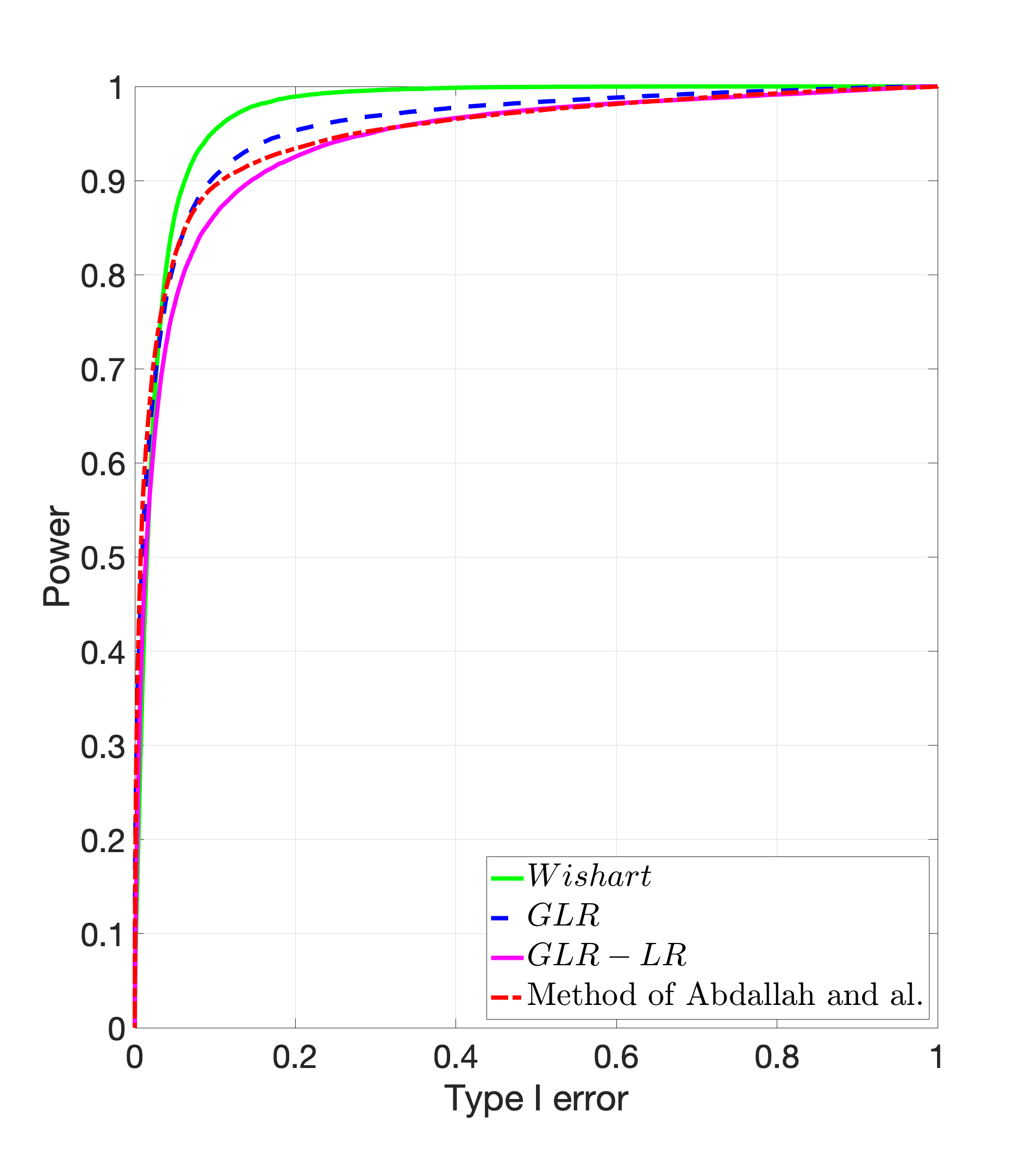}}
    }
    \caption{ROC plots for the two scenes}
    \label{fig:SAR_ROC}
\end{figure}
We observe some improvement of the proposed test statistic for type I errors greater than 15\% for the scene 1 or greater than 5\% for the scene 2. 
  
\section{Conclusion}
\label{section:Conclusion}

In this paper, the problem of covariance equality testing in low-rank Gaussian models has been studied. A new test statistic has been proposed, which is based on the asymptotic behaviour of the largest eigenvalues of certain Wishart matrices in the high-dimensional regime where the dimension of the observations and the number of samples both converge to infinity at the same rate. In particular, it is shown that the proposed statistic has a controlled type I error in the high-dimensional regime. 
Simulations on both synthetic and real datasets  have demonstrated that the proposed test statistic is relevant compared to other alternative approaches.

\appendix

\textit{Notations.} Throughout this Appendix, we use the following notations. For a sequence of random matrices $(\X_n)_{n \geq 1}$, $\X_n = o_{\Pbb}(1)$ denotes the convergence of $(\|\X_n\|_2)$ to 0 in probability, while $\X_n = \Ocal_{\Pbb}(1)$ denotes the tightness of $(\|\X_n\|_2)$, as $n\to\infty$, where $\|.\|_2$ stands for the spectral norm. If $\X$ is a random matrix, we denote by $\X^\circ = \X - \Ebb[\X]$.
Finally, $\Ccal^1(U)$ (resp. $\Ccal^{\infty}_c(U)$) denotes the set of continuously differentiable functions (resp. infinitely differentiable functions with compact support) defined on an open set $U$. 

\subsection{Useful results around the Marcenko-Pastur distribution}

\par In this section, we provide well-known results on the Stieltjes transform 
\begin{equation}
    m(z) = \int_{\Rbb} \frac{\drm \mu(\lambda)}{\lambda - z}, \quad z \in \Cbb\backslash\Rbb,
     \label{eq:st_mp}
\end{equation}
of the Marcenko-Pastur distribution $\mu$ with parameters $(\sigma^2,c)$ defined in \eqref{eq:Marchenko_Pastur_distribution}, having the interval $[x^-,x^+]$ as support with $x^{\pm} = \sigma^2(1\pm\sqrt{c})^2$, and which will be of constant use for the proofs of Theorems \ref{theorem:Spike_model_Limits}, \ref{theorem:CLT_lambda_R_Hat} and Corollary \ref{corollary:CLT_gamma} below.

We first recall that the limit $m(x) = \lim_{z \in \Cbb^+, z \to x} m(z)$ exists for all $x \in \Rbb \backslash\{x^-,x^+\}$, and that for all $z \in \Cbb\backslash\{x^-,x^+\}$, $m(z)$  satisfies the equation:
\begin{equation}
    m(z) = \frac{1+\sigma^2 c m(z)}{\sigma^2 - z \left(1+\sigma^2 c m(z)\right)} = \frac{w(z)}{z(\sigma^2 - w(z))},
    \label{eq:Canonical_Equation}
\end{equation}
with 
\begin{align}
    w(z) = z \left(1+\sigma^2 c m(z)\right).
    \label{def:w}
\end{align}
Moreover, $m$ is continuously differentiable on $\Rbb\backslash\{0,x^-,x^+\}$.
We now provide some results on the function $w$, which plays a central role in describing the behaviour of the largest eigenvalues of $\hat{\R}$. From \eqref{eq:Canonical_Equation}, we observe that for all $z \in \Cbb\backslash\{x^-,x^+\}$,
\begin{equation}
    \phi(w(z)) = z,
    \label{eq:Psi_w}
\end{equation}
where $\phi$ is defined as:
\begin{align}
    \phi(w) = w\left(1-\frac{\sigma^2 c}{\sigma^2 - w}\right).
    \label{eq:phi}
\end{align}
Function $\phi$ is increasing on $(-\infty, w^-) \cup (w^+,\infty)$ and decreasing on $(w^-,\sigma^2) \cup (\sigma^2,w^+)$, with $w^\pm = \sigma^2\left(1\pm\sqrt{c}\right)$ and $\phi(w^\pm) = x^\pm$. 

Next, we consider the following lemma (see \cite{Mestre2008a}) regarding $w$.
\begin{lemma}
    \label{lemma:w}
    For all $x \in \Rbb \backslash\{x^-,x^+\}$, $w(x) \in \phi^{-1}(\{x\})$. Moreover, among the preimages of $x$ under $\phi$, 
    \begin{itemize}
        \item if $x \in \left(x^-,x^+\right)$, $w(x)$ is the unique one such that $\Im(w(x)) > 0$;
        \item if $x \in \Rbb\backslash\left[x^-,x^+\right]$, $w(x)$ is real and is the unique preimage such that $\phi'(w(x)) > 0$. 
    \end{itemize}
    Finally, $z \in \Cbb\backslash\Rbb$ implies that $w(z) \in \Cbb\backslash\Rbb$.
\end{lemma}
Let $\gamma \geq 0$. From Lemma \ref{lemma:w}, it is easily deduced that the equation $w(x) = \gamma + \sigma^2$ admits a solution iff $\gamma > \sigma^2 \sqrt{c}$, and that the solution is unique and given by 
\begin{align}
    x = \phi(\gamma+\sigma^2) = \frac{(\gamma+\sigma^2)(\gamma+\sigma^2 c)}{\gamma}.
\end{align}
Finally, we also have the following result giving various useful formulas. Define $\tilde{m}(z) = -\frac{1}{z(1+\sigma^2 c m(z))}$ and $\tau(z) = z m(z) \tilde{m}(z)$.
\begin{lemma}
    \label{lemma:formulas}
    If $\gamma > \sigma^2\sqrt{c}$, then we have
    \begin{align}
        m(\phi(\gamma+\sigma^2)) &= - \frac{1}{\gamma + \sigma^2 c}, 
        \\
        \tilde{m}(\phi(\gamma+\sigma^2)) &= - \frac{1}{\gamma+\sigma^2},
        \\
        m'(\phi(\gamma+\sigma^2)) &= \frac{\gamma^2}{(\gamma+\sigma^2 c)^2 (\gamma^2 - \sigma^4 c)},
        \\
        \tilde{m}'(\phi(\gamma+\sigma^2)) &= \frac{\gamma^2}{(\gamma+\sigma^2)^2 (\gamma^2 - \sigma^4 c)},
        \\
        \tau(\phi(\gamma+\sigma^2)) &= \frac{1}{\gamma}, 
        \\
        \tau'(\phi(\gamma+\sigma^2)) &= - \frac{1}{\gamma^2 - \sigma^4 c}.
    \end{align}
\end{lemma}

\subsection{Proof of Theorem \ref{theorem:Spike_model_Limits}}
\label{section:Proof_Theorem_Spike_Model_Limits}

This proof is based on techniques developed in \cite{Benaych-Georges2011a}.

\subsubsection{Some notations} Denote by $\e_1,\ldots,\e_N$ the column vectors of the standard basis of $\Cbb^N$, and let
\begin{align}
  \J_1 = \sum_{n=1}^{N_1} \e_n\e_n^*,
\end{align}
and for $\ell=2,\ldots,L$,
\begin{align}
  \J_{\ell} = \sum_{n=N_1+\ldots+N_{\ell-1}+1}^{N_1 + \ldots + N_{\ell}} \e_n\e_n^*.
\end{align}
We also consider the following eigendecomposition of $\Gammabs_{\ell}$
\begin{align}
  \Gammabs_{\ell} = \U_{\ell} \D_{\ell} \U_{\ell}^*,
\end{align}
with $\U_{\ell}$ a $M \times K$ isometric matrix and $\D_{\ell} = \diag\left(\lambda_1(\Gammabs_{\ell}),\ldots,\lambda_K(\Gammabs_{\ell})\right)$.

\subsubsection{Linearization} Let $\Y$ be the $M \times N$ matrix given by $\Y = \left[\y_{1,1},\ldots,\y_{N_1,1}, \ldots, \y_{1,L},\ldots,\y_{N_L,L}\right]$. Due to the Gaussian model, one can assume, without loss of generality, that
\begin{align}
  \Y = \Omegabs \S^* + \W,
\end{align}
where $\W$ is $M \times N$ matrix with i.i.d. $\Ncal_{\Cbb}(0,\sigma^2)$ entries and where
$\Omegabs = \left[\U_1 \D_1^{\frac{1}{2}},\ldots,\U_L \D_L^{\frac{1}{2}}\right]$
and $\S = \left[\S_1,\ldots,\S_L\right]$.
with $\S_{\ell} = \J_{\ell} \X$ and $\X$ a $N \times K$ matrix with i.i.d. $\Ncal_{\Cbb}(0,1)$ entries and independent of $\W$. In particular, we have $\hat{\R} = \frac{1}{N}\Y\Y^*$, and $\Y$ is a fixed rank (at most $KL$) perturbation of $\W$ so that from \textit{Weyl's} inequality and the classical results on the extreme eigenvalues of \textit{Wishart} matrices (see e.g. \cite{Geman1980}), it holds that
\begin{equation}
    \lambda_M(\hat{\R}) \xrightarrow[M\to\infty]{a.s.} x^-,
\end{equation}
while \textit{a.s.}
\begin{equation}
    \limsup_{M\to\infty} \lambda_{KL+1}\left(\hat{\R}\right)     
    \leq
    \limsup_{M\to\infty} \lambda_{1}\left(\frac{1}{N}\W\W^*\right)
    = x^+.
\end{equation}
To study the remaining eigenvalues of $\hat{\R}$, we use the linearization trick which consists in studying the following Hermitian block matrix
\begin{equation}
    \check{\Y} =
    \begin{bmatrix}
        \mathbf{0} & \frac{1}{\sqrt{N}} \Y
        \\
        \frac{1}{\sqrt{N}} \Y^* & \mathbf{0}
    \end{bmatrix},
\end{equation}
for which it is well-known that \cite[Th. 7.3.7]{horn2012matrix}
\begin{equation}
  \mathrm{sp}\left(\check{\Y}\right) = \left\{\pm \sqrt{\lambda_k\left(\hat{\R}\right)}\right\} \cup \{0\}.
\end{equation}

\subsubsection{Asymptotics for the characteristic polynomial of \texorpdfstring{$\check{\Y}$}{Y}} Obviously, we have:

\begin{equation}
    \check{\Y} = \B \check{\I} \B^* + \check{\W},   
\end{equation}
where $\B$, $\check{\I}$ and $\check{\W}$ are given by:
\begin{equation}
    \B =
    \begin{bmatrix}
        \mathbf{\Omega} & \mathbf{0}
        \\
        \mathbf{0} & \frac{1}{\sqrt{N}} \mathbf{S}
    \end{bmatrix}, 
    \check{\mathbf{I}} =
    \begin{bmatrix}
        \mathbf{0} & \mathbf{I}
        \\
        \mathbf{I} & \mathbf{0}
    \end{bmatrix},
    \check{\mathbf{W}} =
    \begin{bmatrix}
        \mathbf{0} & \frac{1}{\sqrt{N}} \mathbf{W}
        \\
        \frac{1}{\sqrt{N}} \mathbf{W}^* & \mathbf{0}
    \end{bmatrix}.
\end{equation}
Let $\epsilon > 0$ and let $\Dcal_{\epsilon}$ the $\epsilon$-neighborhood in $\Cbb$ of the set $\Dcal = \left[x^-,x^+\right]$. Let $\Kcal$ be a compact subset of $\Cbb \backslash \left(\Dcal_{\epsilon} \cup (-\infty,0)\right)$.  Then (see again \cite{Geman1980}), with probability one for all large $M$,
\begin{equation}
  \mathrm{sp}\left(\check{\W}\right) \cap \left\{\sqrt{z}: z \in \Kcal\right\} = \emptyset,
\end{equation}
and the following factorization
\begin{equation}
    \det\left(\check{\Y} - \sqrt{z} \I\right) = \det\left(\check{\W} - \sqrt{z} \I\right) \det\left(\check{\I}\right) \hat{P}(z),
    \label{eq:Determinant_P_z}
\end{equation}
where
\begin{equation}
    \hat{P}(z) = \det\left(\check{\I} + \hat{\mathbf{\Xi}}(z)\right),
\end{equation}
and $\hat{\mathbf{\Xi}}(z) = \B^*\left(\check{\W} - \sqrt{z}\I\right)^{-1}\B$, holds for all $z \in \Kcal$. Then from the block matrix inversion formula, we have
\begin{equation}
\label{eq:chi}
    \hat{\mathbf{\Xi}}(z)
    =
    \begin{bmatrix}
        \sqrt{z} \mathbf{\Omega}^* \mathbf{Q}(z) \mathbf{\Omega} & \frac{1}{N}\mathbf{\Omega}^*\mathbf{Q}(z)\mathbf{W}\mathbf{S}
        \\
        \frac{1}{N}\mathbf{S}^*\mathbf{W}^*\mathbf{Q}(z) \mathbf{\Omega} & \sqrt{z} \frac{1}{N} \mathbf{S}^*\tilde{\mathbf{Q}}(z)\mathbf{S}
    \end{bmatrix},
\end{equation}
where $\Q(z)$ and $\tilde{\Q}(z)$ are the resolvent matrices of $\frac{1}{N}\W\W^*$ and $\frac{1}{N}\W^*\W$ given by
\begin{equation}
    \Q(z) = \left(\frac{1}{N}\W\W^* - z \I\right)^{-1} \text{, }\tilde{\Q}(z) = \left(\frac{1}{N}\W^*\W - z \I\right)^{-1}.
\end{equation}
We then use the following result.
\begin{proposition}
  \label{proposition:quad_form}
  We have
  \begin{equation}
    \sup_{z \in \Kcal}\left\|\Omegabs^*\Q(z)\Omegabs - m(z) \Omegabs^*\Omegabs\right\|_2 \xrightarrow[M\to\infty]{a.s.} 0,
  \end{equation}
  as well as
  \begin{equation}
    \sup_{z \in \Kcal}\left\|\frac{1}{N}\S_{k}^*\tilde{\Q}(z)\S_{\ell} + \frac{\delta_{k-\ell} \frac{N_k}{N}}{z\left(1+\sigma^2 c m(z)\right)} \I\right\|_2 \xrightarrow[M\to\infty]{a.s.} 0,
  \end{equation}
  and
  \begin{equation}
    \sup_{z \in \Kcal} \left\|\frac{1}{N}\S_{k}^*\W^*\Q(z)\Omegabs\right\|_2 \xrightarrow[M\to\infty]{a.s.} 0.
  \end{equation}
\end{proposition}
\begin{proof}
    Proposition \ref{proposition:quad_form} is obtained as a trivial modification of standard results in random matrix theory regarding quadratic forms of resolvents of standard Wishart matrices (see e.g. \cite[Sec. 5.5]{Hachem2012}) and the proof is therefore omitted.
\end{proof}
Using Proposition \ref{proposition:quad_form}, we deduce that:
\begin{equation}
    \sup_{z \in \Kcal} \left\|\hat{\Xibs}(z) - \Xibs(z)\right\|_2 \xrightarrow[M\to\infty]{a.s.} 0,
\end{equation}
where
\begin{equation}
    \Xibs(z) = 
    \begin{bmatrix}
        \sqrt{z} m(z) \Omegabs^*\Omegabs & \mathbf{0}
        \\
        \mathbf{0} & \A(z)
    \end{bmatrix},
\end{equation}
with $\A(z)$ the $KL \times KL$ block diagonal matrix given by
\begin{equation}
  \A(z) =
  - \frac{1}{\sqrt{z}\left(1+\sigma^2 c m(z)\right)}
  \begin{bmatrix}
    \frac{N_1}{N} \I & &
    \\
    & \ddots & &
    \\
    & & \frac{N_L}{N} \I
  \end{bmatrix}.
\end{equation}
It is straightforward to check that
\begin{align}
 \det\left(\check{\I} + \Xibs(z)\right)
  &= \det\left(\sqrt{z} m(z) \Omegabs^*\Omegabs \A(z) - \I\right)
  \notag\\
  &= \det\left(-\frac{m(z)}{1+\sigma^2 c m(z)} \Gammabs - \I\right),
\end{align}
where $\Gammabs$ is defined in \eqref{eq:Gamma}. Consequently, from Assumption \ref{assumption:Convergence_gamma}, one has
\begin{equation}
    \sup_{z \in \Kcal} \left|\hat{P}(z) - P(z)\right| \xrightarrow[N\to\infty]{a.s.} 0,
    \label{eq:Convergence_Phat}
\end{equation}
where $P(z) = \prod_{k=1}^{KL} \left(-\frac{m(z)}{1+\sigma^2 c m(z)} \gamma_k - 1\right)$.
Using the equation \eqref{eq:Canonical_Equation}, one can rewrite $P(z)$ as $P(z) = \prod_{k=1}^{KL} \left(\frac{\gamma_k}{w(z)-\sigma^2} - 1\right)$, with $w$ defined in \eqref{def:w}.

\subsubsection{Spectrum of \texorpdfstring{$\hat{\mathbf{R}}$}{R}} Using Lemma \ref{lemma:w} and the discussion below, we immediately obtain that the set of zeros of $P$ is given by
\begin{equation}
    \Zcal = \left\{\phi(\gamma_k+\sigma^2) : k=1,\ldots,KL, \ \gamma_k > \sigma^2\sqrt{c}\right\}.
\end{equation}
Moreover, Lemma \ref{lemma:w} also implies that $P$ has no pole and thus $P$ is holomorphic on $\Cbb \backslash [x^-,x^+]$. 

Let $Q = |\Zcal|$ and denote by $x_1 > \ldots > x_Q$ the elements of $\Zcal$. We also set $\epsilon > 0$ small enough such that
\begin{equation}
    \Dcal_{\epsilon} \cap \bigcap_{q=1}^Q [x_q-\epsilon,x_q+\epsilon] = \emptyset.
\end{equation}
For all $q =1,\ldots,Q$, let $\Ccal_q$ be a continuously differentiable simple closed curve intersecting the real axis only at the two points $x_q \pm \epsilon$ and enclosing $x_q$ so that $\Ccal_q$ is a compact subset of $\Cbb \backslash \left(\Dcal_{\epsilon} \cup (-\infty,0)\right)$. Applying \eqref{eq:Convergence_Phat} with $\Kcal = \Ccal_q$ provides that with probability one for all large $M$,
\begin{equation}
    \left|\hat{P}(z) - P(z)\right| < \left|P(z)\right|,
\end{equation}
for all $z \in \Ccal_q$, with both $\hat{P}$ and $P$ being holomorphic on any open set enclosed by $\Ccal_q$. Thus, for all $q=1,\ldots,Q$, we deduce from Rouch\'e's Theorem that $\hat{P}$ admits a unique zero in the interval $[x_q-\epsilon, x_q+\epsilon]$. With a similar reasoning, $\hat{P}$ does not have any zero in $(0, x^- - \epsilon) \cup (x_1 + \epsilon , \infty)$ with probability one for all large $M$. Therefore, getting back to \eqref{eq:Determinant_P_z} and since $\epsilon$ can be made arbitrarily small, it follows that
\begin{equation}
    \lambda_k\left(\hat{\R}\right) \xrightarrow[M\to\infty]{a.s.} \phi(\gamma_k+\sigma^2) = \frac{(\gamma_k + \sigma^2)(\gamma_k + \sigma^2 c)}{\gamma_k},
\end{equation}
for all $k=1,\ldots,KL$ such that $\gamma_k > \sigma^2 \sqrt{c}$. Moreover, with probability one for all large $M$, $\lambda_k(\hat{\R}) \in \Dcal_{\epsilon}$ for all $k=1,\ldots,KL$ such that $\gamma_k \leq \sigma^2 \sqrt{c}$. Since the empirical spectral distribution $\hat{\mu}$ of $\hat{\R}$ converges a.s. to the Marcenko-Pastur distribution as $M\to\infty$,
this further implies that for all $k=1,\ldots,KL$ such that $\gamma_k \leq \sigma^2 \sqrt{c}$, $\lambda_k\left(\hat{\R}\right) \xrightarrow[M\to\infty]{a.s.} x^+$ and similarly $\lambda_{KL+1}\left(\hat{\R}\right) \xrightarrow[M\to\infty]{a.s.} x^+$.

\subsection{{Proof of Theorem \ref{theorem:CLT_lambda_R_Hat}}}
\label{section:Proof_Theorem_CLT_lambda_R_Hat}

\subsubsection{Some notations.} Let us first recall that under hypothesis $\Hcal_0$, we have $\Gammabs_1 = \ldots = \Gammabs_{L} = \Gammabs$, and we denote by $\Gammabs = \U\D\U^*$ its eigendecomposition with $\U$ a $M\times K$ isometric matrix and $\D = \diag\left(\lambda_1(\Gammabs),\ldots,\lambda_K(\Gammabs)\right)$. To unify some notations, define as in the previous section $\Y_{\ell} = [\y_{1,\ell},\ldots,\y_{N_{\ell},\ell}]$ for all $\ell=1,\ldots,L$ so that we have
\begin{align}
    \Y_{\ell} = \Omegabs\S_{\ell}^* + \W_{\ell},
\end{align}
where $\Omegabs = \U\D^{1/2}$, $\S_{1},\ldots,\S_L$ are independent matrices such that $\S_{\ell} = [\s_{1,\ell},\ldots,\s_{K,\ell}]$ is $N_{\ell} \times K$ with i.i.d. $\Ncal_{\Cbb}(0,1)$ entries, and where $\W_1,\ldots,\W_L$ are independent matrices with $\W_{\ell}$ having i.i.d. $\Ncal_{\Cbb}(0,\sigma^2)$ entries. We also define $\Y_0 = [\Y_1,\ldots,\Y_L]$ so that
\begin{align}
    \Y_0 = \Omegabs \S_0^* + \W_0,
\end{align}
with $\S_0 = [\S_1^*,\ldots,\S_L^*]^*$ and $\W_0 = [\W_1,\ldots,\W_L]$, and write $N_0 = N_1 + \ldots + N_L$, so that $\hat{\R}_0 = \frac{\Y_0\Y_0^*}{N_0} = \hat{\R}$. Moreover, let $c_0 = c = (\frac{1}{c_1} + \ldots + \frac{1}{c_L})^{-1}$ and
\begin{equation}
    a = \sigma^2\min_{\ell=0,\ldots,L} \left(1-\sqrt{c_{\ell}}\right)^2,
    \quad
    b = \sigma^2
    \max_{\ell=0,\ldots,L} \left(1+\sqrt{c_{\ell}}\right)^2,
\end{equation}
and consider $\varphi \in \Ccal_c^{\infty}(\Rbb)$ such that $\supp(\varphi) = \left[a-\epsilon,b+\epsilon\right]$ and $\varphi(t) = 1$ for all $t \in \left[a-\frac{\epsilon}{2},b+\frac{\epsilon}{2}\right]$, where $\epsilon < a$ . The following quantity defined as
\begin{equation}
    \chi = \prod_{\ell=0}^L \det \varphi\left(\frac{\W_{\ell}\W_{\ell}^*}{N_{\ell}}\right),
    \label{def:chi}
\end{equation}
verifies $\chi = 1$ with probability 1 for all large $M$ from the classical results on the localization of the eigenvalues of \textit{Wishart} matrices \cite{Geman1980}. Recall also the definition of $m$ and $w$ in \eqref{eq:st_mp} and \eqref{def:w} respectively and denote for all $\ell=0,\ldots,L$ by $m_{\ell}$ the \textit{Stieltjes} transform of the \textit{Marcenko-Pastur} distribution with parameter $(c_{\ell},\sigma^2)$, as well as for all $z \in \Cbb\backslash [x_{\ell}^-,x_{\ell}^+]$
\begin{align}
    w_{\ell}(z) &= z \left(1+\sigma^2 c_{\ell} m_{\ell}(z)\right),
    \\
    \tilde{m}_{\ell}(z) &= - \frac{1}{z(1+\sigma^2 c_{\ell} m_{\ell}(z))},
    \\
    \tau_{\ell}(z) &= z m_{\ell}(z) \tilde{m}(z),
\end{align}
with $x_{\ell}^{\pm} = \sigma^2\left(1 \pm \sqrt{c_{\ell}}\right)^2$.

\subsubsection{Characteristic Polynomials Approximation} 

The first step of the proof consists in using the trick from \cite{Benaych-Georges2012} whose main idea is to relate the cumulative distribution function of the spiked eigenvalues with the determinant of certain random matrices.

Using Theorem \ref{theorem:Spike_model_Limits} and the same arguments used to obtain the factorization \eqref{eq:Determinant_P_z} and \eqref{eq:chi} in Appendix \ref{section:Proof_Theorem_Spike_Model_Limits}, we have that $\lambda_1(\hat{\R}_\ell),\ldots,\lambda_K(\hat{\R}_{\ell})$ are the zeros of
\begin{equation}
    \hat{P}_\ell(z) = \det\left(\check{\I} + \hat{\Xibs}_{\ell}(z)\right),
\end{equation}
for all $\ell \in \{0, \dots, L\}$, with probability one for all large $M$, with 
\begin{equation}
    \hat{\Xibs}_{\ell}(z) = 
    \begin{bmatrix}
        \sqrt{z} \Omegabs^* \Q_{\ell}(z) \Omegabs \chi & \frac{1}{N_{\ell}} \Omegabs^* \Q_{\ell}(z) \W_{\ell} \S_{\ell} \chi
        \\
        \frac{1}{N_{\ell}} \S_{\ell}^*\W_{\ell}^*\Q_{\ell}(z) \Omegabs \chi & \sqrt{z} \frac{1}{N_{\ell}} \S_{\ell}^* \tilde{\Q}_{\ell}(z) \S_{\ell} \chi
    \end{bmatrix},
\end{equation}
where $\Q_{\ell}(z) = \left(\frac{\W_{\ell} \W_{\ell}^*}{N_{\ell}} -  z \I\right)^{-1}$
and $\tilde{\Q}_{\ell}(z) = \left(\frac{\W_{\ell}^* \W_{\ell}}{N_{\ell}} -  z \I\right)^{-1}$. For all $\ell,k$, let $-\infty < x_{k,\ell} < y_{k,\ell} < +\infty$ and denote
\begin{align}
    \rho_{k,\ell} = \frac{(\gamma_k + \sigma^2)(\gamma_k+\sigma^2 c_{\ell})}{\gamma_k}.
\end{align}
Then with probability one for all large $M$, we have
\begin{align}
    \sqrt{M} \left(\lambda_k(\hat{\R}_{\ell}) - \rho_{k,\ell}\right)& \in [x_{k,\ell},y_{k,\ell}]
    \notag\\
    \Leftrightarrow 
    \hat{P}_{\ell}\left(\rho_{k,\ell} + \frac{x_{k,\ell}}{\sqrt{M}}\right)&
    \hat{P}_{\ell}\left(\rho_{k,\ell} + \frac{y_{k,\ell}}{\sqrt{M}}\right)
    < 0.
\end{align}
Therefore, as $M \to \infty$,
\begin{align}
    &\Pbb\left( \bigcap_{k=1}^K \bigcap_{\ell=0}^L \left\{\sqrt{M} \left(\lambda_k(\hat{\R}_{\ell}) - \rho_{k,\ell}\right) \in [x_{k,\ell},y_{k,\ell}]\right\}\right)
    = \notag\\
    & \Pbb\left( \bigcap_{k=1}^K\bigcap_{\ell=0}^L \left\{\hat{P}_{\ell}\left(\rho_{k,\ell} + \frac{x_{k,\ell}}{\sqrt{M}}\right)
    \hat{P}_{\ell}\left(\rho_{k,\ell} + \frac{y_{k,\ell}}{\sqrt{M}}\right)
    < 0\right\}\right)\notag\\
    &\quad + o(1).
    \label{eq:prod_det}
\end{align}
The following proposition provides the expansion of $\hat{P}_{\ell}\left(\rho_{k,\ell} + \frac{x}{\sqrt{M}}\right)$ around $\rho_{k,\ell}$.
\begin{proposition}
\label{prop:det}
 For all $x \in  \Rbb$,
 % \vspace{-0.25cm}
  \begin{align}
    &\hat{P}_{\ell}\left(\rho_{k,\ell} + \frac{x}{\sqrt{M}}\right)
    = \frac{1}{\sqrt{M}} \prod_{i \neq k} \left(\gamma_i \tau_{\ell}(\rho_{k,\ell}) -1 \right)
    \notag\\
    & \times
    \Biggl(
        x \gamma_k \tau'_{\ell}(\rho_{k,\ell}) - 2 \sqrt{\gamma_k}\Re\left(\eta_{3,k,\ell}\right)
        +  \gamma_k\rho_{k,\ell} \tilde{m}_{\ell}(\rho_{k,\ell}) \eta_{1,k,\ell}
        \notag\\
        & \qquad\qquad + \gamma_k \rho_{k,\ell} m_{\ell}(\rho_{k,\ell}) \eta_{2,k,\ell})
        \Biggr)
      + o_{\Pbb}\left(\frac{1}{\sqrt{M}}\right),
        \label{eq:dev_pol_car}
\end{align}
where $\tau_{\ell}(z) = z m_{\ell}(z) \tilde{m}_{\ell}(z)$ and
\begin{align}
    \eta_{1,k,\ell} &= \sqrt{M}\u_k^* \left(\Q_{\ell}(\rho_{k,\ell}) \chi\right)^\circ \u_k,
    \\
    \eta_{2,k,\ell} &= \frac{\sqrt{M}}{N_{\ell}} \left(\s_{k,\ell}^* \tilde{\Q}_{\ell}(\rho_{k,\ell}) \s_{k,\ell}\chi\right)^\circ,
    \\
    \eta_{3,k,\ell} &= \frac{\sqrt{M}}{N_{\ell}} \u_k^* \Q_{\ell}(\rho_{k,\ell})\W_{\ell}\s_{k,\ell} \chi.
\end{align}
\end{proposition}
\begin{proof} 
    The proof is deferred to Appendix \ref{section:proof_lemma_det}.
\end{proof}
From \eqref{eq:prod_det} and Proposition \ref{prop:det}, it is clear that we have to study a CLT for the following generic quantity
\begin{equation}
  \eta =
  \sum_{\ell=0}^{L} \sum_{k=1}^K \left(\beta_{1,k,\ell} \eta_{1,k,\ell} + \beta_{2,k,\ell} \eta_{2,k,\ell} + \Re\left(\overline{\beta_{3,k,\ell}}  \eta_{3,k,\ell}\right)\right),
  \label{eq:def_eta}
\end{equation}
where $\left(\beta_{i,k,\ell}\right)_{\substack{i = 1,2 \\ k=1,\ldots,K \\ \ell = 0,\ldots,L}} \in \Rbb^{2K(L+1)}$ and $\left(\beta_{3,k,\ell}\right)_{\substack{k=1,\ldots,K \\ \ell = 0,\ldots,L}} \in \Cbb^{K(L+1)}$.

\subsubsection{{Central Limit Theorem}} Let us consider the characteristic function $\Psi(u) = \mathbb{E}\left[\xi(u)\right]$, with $\xi(u) = \exp\left(\mathrm{i} u \eta\right)$. Our approach consists in deriving a perturbed differential equation for $\Psi$ as shown in the following proposition.
Let $\mathrm{\bdiag}()$ denotes the block diagonal operator. Define $\K = \bdiag\left(\K_1,\ldots,\K_K\right)$ with $\K_k = \bdiag(\K_{1,k},\ldots,\K_{4,k})$ and where $(\K_{i,k})_{i=1,\ldots,4}$ are $(L+1)\times(L+1)$ symmetric matrices with entries given by
\begin{align}
    &[\mathbf{K}_{1, k}]_{\ell+1,\ell'+1} = 
    \notag\\
    &\begin{cases}
        \frac{\sigma^4 c_{\ell}}{(\gamma_k + \sigma^2 c_{\ell})^2 (\gamma_k^2 - \sigma^4 c_{\ell})} & \text{if } \ell = \ell'
        \\[10pt]
        \frac{\sigma^4 c_0}{(\gamma_k + \sigma^2 c_0)(\gamma_k + \sigma^2 c_{\ell'}) (\gamma_k^2 - \sigma^4 c_0)}
        & \text{if } \ell = 0 < \ell'
        \\[10pt]
        0 & \text{if } 0 < \ell < \ell'
    \end{cases},
\end{align}
\begin{align}
    [\mathbf{K}_{2, k}]_{\ell+1,\ell'+1} &= 
    \begin{cases}
        \frac{c_{\ell}}{(\gamma_k + \sigma^2)^2 (\gamma_k^2 - \sigma^4 c_{\ell})} & \text{if } \ell = \ell'
        \\[10pt]
        \frac{c_0}{(\gamma_k + \sigma^2)^2 (\gamma_k^2 - \sigma^4 c_0)}
        & \text{if } \ell = 0 < \ell'
        \\[10pt]
        0 & \text{if } 0 < \ell < \ell'
    \end{cases},
\end{align}
and for $i \in \{3,4\}$,
\begin{align}
    [\mathbf{K}_{i, k}]_{\ell+1,\ell'+1} &= 
    \begin{cases}
        \frac{1}{2}\frac{\sigma^2 c_{\ell}}{\gamma_k^2 - \sigma^4 c_{\ell}} & \text{if } \ell = \ell'
        \\[10pt]
        \frac{1}{2}\frac{\sigma^2 c_0}{\gamma_k^2 - \sigma^4 c_0}
        & \text{if } \ell = 0 < \ell'
        \\[10pt]
        0 & \text{if } 0 < \ell < \ell'
    \end{cases} 
\end{align}
Denote also $\betabs = \left(\betabs_1^T,\ldots,\betabs_K^T\right)^T$ with
\begin{align}
    &\betabs_k = 
    \Bigl(\beta_{1,k,0} , \ldots, \beta_{1,k,L}, \beta_{2,k,0},\ldots,\beta_{2,k,L},
    \notag\\
     &\quad
     \Re(\beta_{3,k,0}),\ldots,\Re(\beta_{3,k,L}),\Im(\beta_{3,k,0}),\ldots,\Im(\beta_{3,k,L})\Bigr)^T.
\end{align}
\begin{proposition}
    \label{prop:perturbed_eq_diff}
   The matrix $\K$ is positive definite and
    \begin{align}
      \Psi'(u) = - u \betabs^T \K \betabs \Psi(u) + \frac{\Delta(u)}{\sqrt{M}},
        \label{eq:Equation_Differentielle}
    \end{align}
    where $\Delta$ is a continuous function such that $|\Delta(u)|< \Prm(u)$ for some polynomial $\Prm$ with positive coefficients independent of $M$.
\end{proposition}
\begin{proof}
    The proof is deferred to Appendix \ref{section:proof_prop_funcar}.
\end{proof}
From Proposition \ref{prop:perturbed_eq_diff}, by solving the perturbed differential equation in \eqref{eq:Equation_Differentielle}, we deduce that
\begin{equation}
    \Psi(u) \xrightarrow[M\to\infty]{} \exp\left(- \betabs^T \K \betabs \frac{u^2}{2 }\right),
\end{equation}
which of course implies that
\begin{equation}
    \eta \xrightarrow[M\to\infty]{\mathcal{D}} \mathcal{N}_{\Rbb}\left(0,\betabs^T \K \betabs\right).
    \label{eq:TCL_eta}
\end{equation}
The final step of the proof consists in transferring the CLT  to the $K$ largest eigenvalues of $(\hat{\R}_{\ell})_{\ell=1,\ldots,L}$. From Proposition \eqref{prop:det}, we have that:
\begin{align}
    &\hat{P}_{\ell}\left(\rho_{k,\ell} + \frac{x}{\sqrt{M}}\right)
    =
    \notag\\
    &\qquad\frac{1}{\sqrt{M}} \gamma_k \tau'_{\ell}(\rho_{k,\ell}) \left( x  - \zeta_{k,\ell} + o_{\Pbb}\left(1\right)\right) \prod_{i \neq k} \left(\gamma_i \tau_{\ell}(\rho_{k,\ell}) -1 \right),
\end{align}
with
\begin{align}
    \zeta_{k,\ell} &=
    \frac{1}{\gamma_k \tau'_{\ell}(\rho_{k,\ell})} \Bigl(2 \sqrt{\gamma_k}\Re\left(\eta_{3,k,\ell}\right)
    - \gamma_k \rho_{k,\ell} \tilde{m}_{\ell}(\rho_{k,\ell})\eta_{1,k,\ell} 
    \notag\\
    &\qquad\qquad - \gamma_k \rho_{k,\ell} m_{\ell}(\rho_{k,\ell})\eta_{2,k,\ell}\Bigr).
\end{align}
Thus, going back to \eqref{eq:prod_det}, we get
\begin{align}
    &\Pbb\left( \bigcap_{k=1}^K \bigcap_{\ell=0}^L \left\{\sqrt{M} \left(\lambda_k(\hat{\R}_{\ell}) - \rho_{k,\ell}\right) \in [x_{k,\ell},y_{k,\ell}]\right\}\right)
    =
    \notag\\
    &\quad \Pbb
    \left(
        \bigcap_{k=1}^K\bigcap_{\ell=0}^L 
        \left\{x_{k,\ell} < \zeta_{k,\ell} + o_{\Pbb}(1) < y_{k,\ell}\right\}
    \right) 
  + o(1).
\end{align}
Using \eqref{eq:TCL_eta}  with suitable values for $\betabs$ as well as the equalities of Lemma \ref{lemma:formulas}, we have $\zetabs = \left(\zeta_{k,\ell}\right)_{\substack{\ell=0,\ldots,L \\ k=1,\ldots,K}} \xrightarrow[M \to \infty]{\Dcal} \Ncal_{\Rbb^{K(L+1)}}\left(\mathbf{0},\Thetabs\right)$,
where $\Thetabs$ is given in the statement of Theorem \ref{theorem:CLT_lambda_R_Hat}. Finally, noticing  that
\begin{align}
    \det(\Thetabs) 
    &=
    \prod_{k=1}^K \prod_{\ell=1}^L \theta_{k,\ell}^2 \left(\theta_0^2 - \sum_{\ell'=1}^L \frac{\vartheta_{k,\ell'}^2}{\theta_{k,\ell}^2}\right)
    \notag\\
    &=
    \left(\sigma^4 c_0^2 (L-1)\right)^{K} \prod_{k=1}^K \left(\frac{\gamma_k+\sigma^2}{\gamma_k}\right)^{2(L+1)}
    \notag\\
    &\hspace{4.5cm}\times\prod_{\ell=1}^L c_{\ell} (\gamma_k^2-\sigma^4 c_{\ell}),
\end{align}
we obtain that $\det(\Thetabs) > 0$ thanks to Assumption  \ref{assumption:Convergence_gamma_Separation_Condition} which concludes the proof of Theorem \ref{theorem:CLT_lambda_R_Hat}.

\subsection{Proof of Corollary \ref{corollary:CLT_gamma}}
\label{section:Delta_Method}

Denote $c_0=c$ and for all $\ell=0,\ldots,L$, let $\hat{\phi}_{\ell}(w) = w \left(1 - \frac{\hat{\sigma}^2 c_{\ell}}{\hat{\sigma}^2 - w}\right)$.
Denote as well $\hat{\R}_0 = \hat{\R}$ and $\hat{\gamma}_{k,0} = \hat{\gamma}_k$ for ease of reading. Under Assumption \ref{assumption:Convergence_gamma_Separation_Condition}, we first observe from Theorem \ref{theorem:Spike_model_Limits} that
\begin{align}
    \lambda_{k}\left(\hat{\R}_{\ell}\right) \xrightarrow[M\to\infty]{a.s.} 
    \phi_{\ell}(\gamma_k+\sigma^2) = \frac{(\gamma_k+\sigma^2)(\gamma_k+\sigma^2 c_{\ell})}{\gamma_k},
\end{align}
so that $\hat{\lambda}_k(\hat{\R}_{\ell}) > \hat{\sigma}^2 \sqrt{c_{\ell}}$
with probability one for all large $M$, and therefore $\hat{\gamma}_{k,\ell} + \hat{\sigma}^2$ coincides
with the largest solution to the equation $\hat{\phi}_{\ell}(w) = \lambda_k(\hat{\R}_{\ell})$. From Lemma \ref{lemma:w}, we deduce that $\hat{\gamma}_{k,\ell} = \hat{w}_{\ell}\left(\hat{\lambda}_k(\hat{\R}_{\ell}) \right) - \hat{\sigma}^2$.
with probability one for all large $M$, where $\hat{w}_{\ell}(z) = z\left(1+\hat{\sigma}^2 c_{\ell}\hat{m}_{\ell}(z)\right)$ with $\hat{m}_{\ell}$ the Stieltjes transform of the Marcenko-Pastur distribution with parameter $(\hat{\sigma}^2, c_{\ell})$. It is easy to see that $\hat{\sigma}^2 = \sigma^2 + \Ocal_{\Pbb}\left(\frac{1}{M}\right)$ and:
\begin{align}
    \hat{m}_{\ell}\left(\hat{\lambda}_k(\hat{\R}_{\ell}) \right) = 
    m_{\ell}\left(\hat{\lambda}_k(\hat{\R}_{\ell}) \right) + \Ocal_{\Pbb}\left(\frac{1}{M}\right).
\end{align}
Therefore, we deduce that
\begin{align}
    \hat{\gamma}_{k,\ell} = 
    w_{\ell}\left(\hat{\lambda}_k(\hat{\R}_{\ell}) \right) - \sigma^2 + \Ocal_{\Pbb}\left(\frac{1}{M}\right).
\end{align}
As $w_{\ell}(\phi_{\ell}(\gamma_k+\sigma^2)) = \gamma_k + \sigma^2$ and $w_{\ell}$ is differentiable at point $\gamma_k+\sigma^2$, a straightforward use of the delta-method provides 
\begin{align}
        \sqrt{M}
        \left(\hat{\gamma}_{k,\ell} - \gamma_k\right)_{\substack{\ell=0,\ldots,L \\ k = 1, \dots, K}}
        \xrightarrow[M\to\infty]{\mathcal{D}} \mathcal{N}_{\mathbb{R}^{K(L+1)}}\left(\mathbf{0},\G\mathbf{\Theta}\G\right),
\end{align}
where $\G = \diag\left((w'_{\ell}(\phi_{\ell}(\gamma_k+\sigma^2)))_{\substack{\ell=0,\ldots,L \\ k = 1, \dots, K}}\right)$. Noticing that $w'_{\ell}(\phi_{\ell}(\gamma_k+\sigma^2)) = \frac{\gamma_k^2}{\gamma_k^2 - \sigma^2 c_{\ell}}$ from Lemma \ref{lemma:formulas}, we end up with $\G\mathbf{\Theta}\G = \Omegabs$, where $\Omegabs$ is given in the statement of Corollary \ref{corollary:CLT_gamma}. Another immediate use of the delta-method allows to transfer the CLT from $\left(\hat{\gamma}_{k,\ell}\right)_{\substack{\ell=0,\ldots,L \\ k = 1, \dots, K}}$ to $\hat{\gammabs}$.

\subsection{Additional proofs}
\label{section:add_proofs}

\subsubsection{Proof of Proposition \ref{prop:det}}
\label{section:proof_lemma_det}

It is easy to see using Proposition \ref{proposition:quad_form} that for all $x \in \Rbb$,
\begin{align}
   &\hat{\Xibs}_{\ell}\left(\rho_{k,\ell} + \frac{x}{\sqrt{M}}\right) =
   \notag\\
   &\qquad\qquad\begin{bmatrix}
        \sqrt{\rho_{k,\ell}} m_{\ell}(\rho_{k,\ell}) \D
        &
        \mathbf{0}
        \\
        \mathbf{0} 
        &
      \sqrt{\rho_{k,\ell}} \tilde{m}_{\ell}(\rho_{k,\ell}) \I
    \end{bmatrix}
    + \Deltabs,
\end{align}
where
\begin{align}
    &\Deltabs =
    \notag\\
    &\begin{bmatrix}
        \sqrt{\rho_{k,\ell}} \Omegabs^* \left(\Q_{\ell}(\rho_{k,\ell})\chi\right)^\circ \Omegabs
        &
        \frac{1}{N_{\ell}} \Omegabs^* \Q_{\ell}(\rho_{k,\ell}) \chi \W_{\ell} \S_{\ell}
        \notag\\
        \frac{1}{N_{\ell}} \S_{\ell}^*\W_{\ell}^*\Q_{\ell}(\rho_{k,\ell}) \chi \Omegabs 
        & \sqrt{\rho_{k,\ell}} \frac{1}{N_{\ell}} \left(\S_{\ell}^* \tilde{\Q}_{\ell}(\rho_{k,\ell}) \S_{\ell} \chi\right)^\circ
    \end{bmatrix}
    \notag\\
    & + \frac{x}{\sqrt{M}} 
    \begin{bmatrix}
        h\left(\rho_{k,\ell}\right) \D
        &
        \mathbf{0}
        \\
        \mathbf{0} 
        &
        \tilde{h}\left(\rho_{k,\ell}\right) \I
    \end{bmatrix}
    + o_{\Pbb}\left(\frac{1}{\sqrt{M}}\right),
\end{align}
with $h_{\ell}(z) = \frac{m_{\ell}(z)}{2\sqrt{z}} + \sqrt{z} m_{\ell}'(z)$ and $\tilde{h}_{\ell}(z) = \frac{\tilde{m}_{\ell}(z)}{2\sqrt{z}} + \sqrt{z} \tilde{m}_{\ell}'(z)$.
Note that $\left\|\Deltabs\right\|_2 = \Ocal_{\Pbb}\left(\frac{1}{\sqrt{M}}\right)$ and we also consider the partition $\Deltabs = (\Deltabs_{i,j})_{i,j=1,2}$ where each block $\Deltabs_{i,j}$ is $K \times K$. Consider the event $\Acal = \left\{\left\|\Deltabs_{2,2}\right\|_2 < \sqrt{\rho_{k,\ell}} \tilde{m}_\ell(\rho_{k,\ell}) \right\}$. From the block matrix determinant formula, we have:
\begin{align}  
    &\hat{P}_{\ell}\left(\rho_{k,\ell} + \frac{x}{\sqrt{M}}\right) \mathbb{1}_{\Acal}
    = 
    \Phi \ \det\left(\sqrt{\rho_{k,\ell}} \tilde{m}_{\ell}(\rho_{k,\ell}) \I + \Deltabs_{2,2}\right) \mathbb{1}_{\Acal},
\end{align}
with
\begin{align}
    &\Phi = 
    \det
    \biggl(
    \sqrt{\rho_{k,\ell}} m_{\ell}(\rho_{k,\ell}) \D + \Deltabs_{1,1} 
    \notag\\
    &  - (\I+\Deltabs_{2,1}) \left(\sqrt{\rho}_{k,\ell} \tilde{m}_\ell(\rho_{k,\ell}) \I + \Deltabs_{2,2} \right)^{-1} (\I+\Deltabs_{1,2}) \biggr) \mathbb{1}_{\Acal}.
\end{align}
Moreover,
\begin{align}
    &\biggl(\sqrt{\rho}_{k,\ell} \tilde{m}_\ell(\rho_{k,\ell}) \I + \Deltabs_{2,2}\biggr)^{-1} \mathbb{1}_{\Acal} 
    =
    \notag\\
    &\biggl(\frac{\I}{\sqrt{\rho_{k,\ell}} \tilde{m}_\ell(\rho_{k,\ell})}
    -
    \frac{\Deltabs_{2,2}}{\rho_{k,\ell} \tilde{m}_\ell(\rho_{k,\ell})^2}\biggr) \mathbb{1}_{\Acal}
    + o_{\Pbb}\left(\frac{1}{\sqrt{M}}\right),
\end{align}
which yields
\begin{align}
    &\Phi = \det\biggl(
    \sqrt{\rho_{k,\ell}} m_{\ell}(\rho_{k,\ell}) \D - \frac{\I}{\sqrt{\rho}_{k,\ell} \tilde{m}_\ell(\rho_{k,\ell}))}
    + \Deltabs_{1,1} 
    \notag\\
    & \ \qquad - \frac{\Deltabs_{1,2} + \Deltabs_{2,1}}{\sqrt{\rho_{k,\ell}} \tilde{m}_\ell(\rho_{k,\ell})}
    +\frac{\Deltabs_{2,2}}{\rho_{k,\ell} \tilde{m}_\ell(\rho_{k,\ell})^2}
    \biggr) \mathbb{1}_{\Acal} + o_{\Pbb}\left(\frac{1}{\sqrt{M}}\right).
\end{align}
Since $\D = \diag(\gamma_1,\ldots,\gamma_K) + o\left(\frac{1}{\sqrt{M}}\right)$ from Assumption \ref{assumption:Convergence_gamma}, and from Lemma \ref{lemma:formulas}, we have
\begin{align}
    \det\left(\sqrt{\rho_{k,\ell}} m_{\ell}(\rho_{k,\ell}) \D - \frac{\I}{\sqrt{\rho_{k,\ell}} \tilde{m}_\ell(\rho_{k,\ell})}\right) = o\left(\frac{1}{\sqrt{M}}\right).
\end{align}
Using the differential of the determinant, we further have
\begin{align}
    &\Phi = \tr \Biggl[
    \mathrm{com}\left(\sqrt{\rho_{k,\ell}} m_{\ell}(\rho_{k,\ell}) \D - \frac{\I}{\sqrt{\rho}_{k,\ell} \tilde{m}_\ell(\rho_{k,\ell})}\right)^T
    \notag\\
    & \quad \times \left(\Deltabs_{1,1}
    - \frac{\Deltabs_{1,2} + \Deltabs_{2,1}}{\sqrt{\rho}_{k,\ell} \tilde{m}_\ell(\rho_{k,\ell})}
    +\frac{\Deltabs_{2,2}}{\rho_{k,\ell} \tilde{m}_\ell(\rho_{k,\ell})^2}\right)
    \Biggr] 
    \mathbb{1}_{\Acal} 
    \notag\\
    & \quad  +o_{\Pbb}\left(\frac{1}{\sqrt{M}}\right),
\end{align}
where $\mathrm{com}()$ denotes the comatrix operation. A direct computation together with Assumption \ref{assumption:Convergence_gamma} provides
\begin{align}
    &\mathrm{com}\left(\sqrt{\rho_{k,\ell}} m_{\ell}(\rho_{k,\ell}) \D - \frac{\I}{\sqrt{\rho_{k,\ell}} \tilde{m}_\ell(\rho_{k,\ell})}\right)
    = 
    \notag\\ % % \vspace{0.25cm}
    & \qquad \frac{\prod_{i \neq k}\left(\gamma_i \rho_{k,\ell} m_{\ell}(\rho_{k,\ell}) \tilde{m}_{\ell}(\rho_{k,\ell}) -1 \right)}{(\sqrt{\rho_{k,\ell}}\tilde{m}_{\ell}(\rho_{k,\ell}))^{K-1}} \e_k \e_k^* + o\left(\frac{1}{\sqrt{M}}\right).
\end{align}
Consequently,
\begin{align}
    &\Phi =  o_{\Pbb}\left(\frac{1}{\sqrt{M}}\right) + \frac{\prod_{i \neq k}\left(\gamma_i \rho_{k,\ell} m_{\ell}(\rho_{k,\ell}) \tilde{m}_{\ell}(\rho_{k,\ell}) -1 \right)}{(\sqrt{\rho_{k,\ell}}\tilde{m}_{\ell}(\rho_{k,\ell}))^{K-1}} 
    \notag\\
    & \quad \times \left[\Deltabs_{1,1}
    - \frac{\Deltabs_{1,2} + \Deltabs_{2,1}}{\sqrt{\rho_{k,\ell}} \tilde{m}_\ell(\rho_{k,\ell})}
    +\frac{\Deltabs_{2,2}}{\rho_{k,\ell} \tilde{m}_\ell(\rho_{k,\ell})^2}\right]_{k,k}
    \mathbb{1}_{\Acal}.
\end{align}
In the same way, 
\begin{align}
    &\det\left(\sqrt{\rho_{k,\ell}} \tilde{m}_{\ell}(\rho_{k,\ell}) \I + \Deltabs_{2,2}\right) 
    = 
     \notag\\
    &\qquad\qquad\qquad\qquad 
    \left(\sqrt{\rho_{k,\ell}} \tilde{m}_{\ell}(\rho_{k,\ell})\right)^K + \Ocal_{\Pbb}\left(\frac{1}{\sqrt{M}}\right),
\end{align}
so that
\begin{align}
    &\hat{P}_{\ell}\left(\rho_{k,\ell} + \frac{x}{\sqrt{M}}\right) \mathbb{1}_{\Acal}
    =
    \notag\\
    &\left[\sqrt{\rho_{k,\ell}} \tilde{m}_{\ell}(\rho_{k,\ell}) \Deltabs_{1,1} + \frac{\Deltabs_{2,2}}{\sqrt{\rho_{k,\ell}} \tilde{m}_{\ell}(\rho_{k,\ell})} -(\Deltabs_{1,2} + \Deltabs_{2,1})\right]_{k,k}
      \notag\\
      &\times \prod_{i \neq k} \left(\gamma_i \rho_{k,\ell} m_{\ell}(\rho_{k,\ell}) \tilde{m}_{\ell}(\rho_{k,\ell}) -1 \right)
      + o_{\Pbb}\left(\frac{1}{\sqrt{M}}\right).
\end{align}
Since $\mathbb{1}_{\Acal} = 1 + o_{\Pbb}(1)$, we also deduce that
\begin{align}
    \hat{P}_{\ell}\left(\rho_{k,\ell} + \frac{x}{\sqrt{M}}\right) \mathbb{1}_{\Acal^c} = o_{\Pbb}\left(\frac{1}{\sqrt{M}}\right),
\end{align}
and using Assumption \ref{assumption:Convergence_gamma} leads to the result of Proposition \ref{prop:det}.

\subsubsection{Proof of Proposition \ref{prop:perturbed_eq_diff}}
\label{section:proof_prop_funcar}

The proof makes use of well-known techniques specific to the Gaussian distribution, namely the \textit{Stein's lemma} and the \textit{Poincaré's inequality} which we recall below and which have already been exploited in e.g. \cite{Hachem2008, Vallet2015, Mestre2017}. Therefore, we only give the main steps of the proof and skip some details of the computations.

\paragraph{Useful tools.} A function $f:\Cbb^n \to \Cbb$ is said to be  $\Ccal^1(\Cbb^n)$ if $f(\z) = \tilde{f}(\Re(\z),\Im(\z))$ with $\tilde{f} \in \Ccal^1(\Rbb^{2n})$. Moreover, we also recall the definition of the standard complex differential operators
\begin{align}
    \frac{\partial f}{\partial z_k}(\z) &=
    \frac{1}{2}\left(\frac{\partial \tilde{f}}{\partial x_k}(\x,\y) - \irm \frac{\partial \tilde{f}}{\partial y_k}(\x,\y) \right)
    \\
    \frac{\partial f}{\partial \overline{z_k}}(\z) &=
    \frac{1}{2}\left(\frac{\partial \tilde{f}}{\partial x_k}(\x,\y) + \irm \frac{\partial \tilde{f}}{\partial y_k}(\x,\y) \right)
\end{align}
with $\x = \Re(\z)$ and $\y = \Im(\z)$.
\begin{lemma}[Stein's lemma]
  Let $\w \sim \Ncal_{\Cbb^n}(\mathbf{0},\I)$ and $f \in \Ccal^1(\Cbb^n)$, assumed polynomially bounded together with its partial derivatives. Then for all $k =1,\ldots,n$,
  \begin{align}
    \mathbb{E}[f(\w) w_k] = \mathbb{E}\left[\frac{\partial f}{\partial \overline{w_k}}(\w)\right],
    \ 
    \mathbb{E}[f(\w) \overline{w_k}] = \mathbb{E}\left[\frac{\partial f}{\partial w_k}(\w)\right].
                                \notag
  \end{align}
\end{lemma}
\begin{lemma}[Poincar\'e inequality]
    Let $\w \sim \Ncal_{\Cbb^n}(\mathbf{0},\I)$ and $f \in \Ccal^1(\Cbb^n)$, assumed polynomially bounded together with its partial derivatives.
    Then,
    \begin{align}
        \Vbb[f(\w)] \leq
        \sum_{k=1}^n
        \left(\mathbb{E}\left|\frac{\partial f}{\partial \overline{w_k}}(\w)\right|^2 + \mathbb{E}\left|\frac{\partial f}{\partial w_k}(\w)\right|^2\right).
    \end{align}
\end{lemma}
For ease of reading, we introduce the following differentiation operators with respect to the entries of the $M \times N_{\ell}$ matrix $\W_\ell$, which will be constantly used in the derivations below, 
\begin{align}
    \partial_{i,j}^{(\ell)} = \frac{\partial}{\partial [\W_{\ell}]_{i,j}},
  \quad
  \overline{\partial}_{i,j}^{(\ell)} = \frac{\partial}{\partial \overline{[\W_{\ell}]_{i,j}}}.
\end{align}
We will also need the following auxiliary result (see \cite{Hachem2012}) related to the quantity $\chi$ defined in \eqref{def:chi}.
\begin{lemma}
    \label{lemma:regularization}
    For all $p \in \Nbb$ and $r \in \Nbb$, $\mathbb{E}\left[\chi^r\right] = 1 + \Ocal\left(\frac{1}{N^p}\right)$ and for  $\ell \in \{0,\ldots,L\}$ and for any $i \in \{1,\ldots,M\}$, $j \in \{1,\ldots,N_{\ell}\}$,
    \begin{align}
        \mathbb{E}\left[\partial^{(\ell)}_{i,j}\chi^r\right] = \Ocal\left(\frac{1}{N^p}\right)
                                                          \text{ and }
        \mathbb{E}\left[\overline{\partial}^{(\ell)}_{i,j}\chi^r\right] = \Ocal\left(\frac{1}{N^p}\right).
    \end{align}
\end{lemma}
Lemma \ref{lemma:regularization} shows in particular that the presence of the regularization term $\chi$ can be removed from expectations, up to an error term of arbitrary polynomial decay.

In the following, $\Delta$ is a generic notation for a continuous function such that $|\Delta(u)|< \Prm(u)$ for some polynomial $\Prm$ with positive coefficients independent of $M$, and whose value may change from one line to another.

\paragraph{Development of \texorpdfstring{$\Psi'$}{Psi}} Write

\begin{align}
    \Psi'(u) = \irm \sum_{k=1}^K  \sum_{\ell=0}^L
    \Ebb
    \Biggl[
        \Biggl(
            &\beta_{1,k,\ell} \eta_{1,k,\ell} 
            + \beta_{2,k,\ell} \eta_{2,k,\ell} \notag\\
            &+ \Re\Biggl(\overline{\beta_{3,k,\ell}} \eta_{3,k,\ell}\Biggr)
        \Biggr)\xi(u)
    \Biggr].
\end{align}
In the following, we only provide some details for the development of $\Ebb[\eta_{1,k,0} \xi(u)]$, as the other terms can be obtained similarly. Using the resolvent identity, we start by writing
\begin{align}
    \mathbb{E} \left[\eta_{1,k,0} \xi(u)\right] =
    \frac{\sqrt{M}}{\rho_{k,0}} 
    \Ebb\left[\left(\u_k^* \Q_0(\rho_{k,0}) \frac{\W_0\W_0^*}{N_0} \u_k  \chi\right)^\circ \xi(u)\right].
\end{align}
Next, we apply Stein's lemma, Poincaré's inequality and Lemma \ref{lemma:regularization} to obtain
\begin{align}
    &\Ebb\left[\u_k^* \Q_0(\rho_{k,0}) \frac{\W_0\W_0^*}{N_0} \u_k \chi \xi(u)\right]
    =
    \notag\\
    &
    \frac
    {
    \irm u \sigma^2\sum_{i,j} \Ebb\left[[\u_k^*\Q_0(\rho_{k,0})]_{i} [\W_{\ell}^*\u_k]_j \chi
    \overline{\partial}_{i,j}^{(0)}\{\eta\} \xi(u) \right]
    }{N_0(1+\alpha_0(\rho_{k,0}))} \notag
    \\
    &\qquad + \frac{\sigma^2  \Ebb\left[\u_k^* \Q_0(\rho_{k,0}) \u_k \chi \xi(u)\right]}{1+\alpha_0(\rho_{k,0})} + \frac{\Delta(u)}{M},
    \label{eq:dev_eta0}
\end{align}
where $\alpha_{\ell}(z) = \Ebb\left[\frac{\sigma^2}{N_{\ell}}\tr \Q_{\ell}(z) \chi\right]$ for all $\ell=0,\ldots,L$. Using the fact that $\alpha_{0}(\rho_{k,0}) = \sigma^2 c_0 m_0(\rho_{k,0}) + \Ocal(\frac{1}{M^2})$, this gives
\begin{align}
    &\Ebb\left[\eta_{1,k,0}\xi(u)\right] = 
    \notag\\
    &\frac
    {
    \irm u \sigma^2 \sqrt{M}\sum_{i,j} \Ebb\left[[\u_k^*\Q_0(\rho_{k,0})]_{i} [\W_{\ell}^*\u_k]_j \chi
    \overline{\partial}_{i,j}^{(0)}\{\eta\} \xi(u) \right]
    }{N_0 \left(\rho_{k,0}(1+\sigma^2 c_0 m_0(\rho_{k,0})) - \sigma^2\right)}
    \notag\\
    &+ \frac{\Delta(u)}{\sqrt{M}}.
\end{align}
Developing further the derivatives and using Lemma \ref{lemma:regularization}, we have
\begin{align}
    &\sum_{i,j} \Ebb\left[[\u_k^*\Q_0(\rho_{k,0})]_{i} [\W_{\ell}^*\u_k]_j \chi
    \overline{\partial}_{i,j}^{(0)}\{\eta_{1,k',\ell}\} \xi(u) \right] = 
    \notag\\
    & \quad -\sqrt{M}
    \Ebb
    \Biggl[
        \u_k^* \Q_0(\rho_{k,0})\Q_{\ell}(\rho_{k',\ell})\u_{k'}
        \u_{k'}^* \Q_{\ell}(\rho_{k',\ell}) 
        \notag\\
        & \qquad \qquad \qquad \times \frac{\W_{\ell}\W_{\ell}^*}{N_{\ell}} \u_k
        \chi \xi(u)
    \Biggr] + \frac{\Delta(u)}{\sqrt{M}}
    \notag\\
    &\quad=
    \sqrt{M}\theta_{k,\ell} \delta_{k-k'} \Psi(u) + \frac{\Delta(u)}{\sqrt{M}},
\end{align}
with 
\begin{align}
    \kappa_{k,\ell} &= 
    \notag\\
    &\begin{cases}
        \frac{\sigma^2 m_0(\rho_{k,0}) m_0'(\rho_{k,0})}{1+\sigma^2 c_0 m_0(\rho_{k,0})} 
        & \text{ if } \ell=0,
        \\[5pt]
        \frac{\sigma^2 m_{\ell}(\rho_{k,\ell}) m_0(\rho_{k,0}) (1+\sigma^2 c_{0} m_{0}(\rho_{k,0}))}
    {\sigma^2 - \rho_{k, \ell}(1+\sigma^2 c_{0} m_{0}(\rho_{k,0})) (1+\sigma^2 c_{\ell} m_{\ell}(\rho_{k,\ell}))}
    &  \text{ if } \ell \geq 1
    \end{cases}
    \\[10pt]
    &= \begin{cases}
        -\frac{\sigma^2 \gamma_k}{(\gamma_k+\sigma^2 c_0)^2(\gamma_k^2-\sigma^2 c_0)}
        & \text{ if } \ell=0
        \\[5pt]
        - \frac{\sigma^2 \gamma_k}{(\gamma_k+\sigma^2 c_0)(\gamma_k+\sigma^2 c_{\ell})(\gamma_k^2-\sigma^2 c_0)}
    &  \text{ if } \ell \geq 1
    \end{cases},
\end{align}
where the second equality in the expression of $\theta_{k,\ell}$ can be obtained with Lemma \ref{lemma:formulas}. Moreover,
\begin{align}
    &\sum_{i,j} \Ebb\left[[\u_k^*\Q_0(\rho_{k,0})]_{i} [\W_{\ell}^*\u_k]_j \chi
    \overline{\partial}_{i,j}^{(0)}\{\eta_{2,k',\ell}\} \xi(u) \right]
    \notag\\
    & \qquad = -\frac{\sqrt{M}}{N_{\ell}^2}
    \Ebb
    \Biggl[
        \u_k^* \Q_0(\rho_{k,0})\W_{\ell}\tilde{\Q}_{\ell}(\rho_{k',\ell})\s_{k',\ell}
        \notag\\
        &\qquad \qquad \qquad  \quad \times \s_{k',\ell}^* \tilde{\Q}_{\ell}(\rho_{k',\ell}) \W_{\ell}^* \u_k
        \chi \xi(u)
    \Biggr] + \frac{\Delta(u)}{\sqrt{M}}
    \notag\\
    & \qquad =\frac{\Delta(u)}{\sqrt{M}},
\end{align}
and
\begin{align}
    &\sum_{i,j} \Ebb\left[[\u_k^*\Q_0(\rho_{k,0})]_{i} [\W_{\ell}^*\u_k]_j \chi
    \overline{\partial}_{i,j}^{(0)}\{\eta_{3,k',\ell}\} \xi(u) \right]
    \notag\\
    &\qquad = -\frac{\sqrt{M}}{N_{\ell}}
    \Ebb
    \Biggl[
        \u_k^* \Q_0(\rho_{k,0})\Q_{\ell}(\rho_{k',\ell})\W_{\ell}\s_{k',\ell} 
        \notag\\
        & \qquad\qquad\qquad \quad \times \u_{k'}^* \Q_{\ell}(\rho_{k',\ell}) \frac{\W_{\ell}\W_{\ell}^*}{N_{\ell}} \u_k
        \chi \xi(u)
    \Biggr] + \frac{\Delta(u)}{\sqrt{M}}
    \notag\\
    &\qquad  =\frac{\Delta(u)}{\sqrt{M}}.
\end{align}
Finally, using again Lemma \ref{lemma:formulas}, we obtain
\begin{align}
    &\mathbb{E} \left[\eta_{1,k,0} \xi(u)\right] 
    \notag\\
    &=\frac{\irm u \sigma^2 c_0 \sum_{\ell=0}^L \beta_{1,k,\ell} \kappa_{k,\ell}}{\rho_{k,0}(1+\sigma^2 c_0 m_0(\rho_{k,0})) - \sigma^2} \Psi(u) + \frac{\Delta(u)}{\sqrt{M}}
    \notag\\
    &= 
    - \irm u \biggl( \frac{\beta_{1,k,0} \sigma^4 c_0}{(\gamma_k+\sigma^2 c_0)^2(\gamma_k^2-\sigma^2 c_0)}
    \notag\\
    &+ \sum_{\ell=1}^L \frac{\beta_{1,k,\ell}\sigma^4 c_0}{(\gamma_k+\sigma^2 c_0)(\gamma_k+\sigma^2 c_{\ell})(\gamma_k^2-\sigma^2 c_0)} \biggr) \Psi(u) + \frac{\Delta(u)}{\sqrt{M}}.
\end{align}
Using similar computations for the remaining terms $(\mathbb{E} \left[\eta_{i,k,\ell} \xi(u)\right])_{\substack{\ell \geq 1 \\ i = 2,3}}$ in $\Psi'(u)$, we finally obtain the result of Proposition \ref{prop:perturbed_eq_diff}.

\bibliographystyle{IEEEtran}
\bibliography{tsp2023}

\end{document}